\documentclass[english,abstract=true,captions=tableheading]{scrartcl}

\usepackage[utf8]{inputenc}
\usepackage[T1]{fontenc}
\usepackage{lmodern}

\usepackage{microtype}

\usepackage[english]{babel}
\usepackage{csquotes}

\usepackage{amsmath,amssymb,amsthm}
\usepackage[backend=biber,style=numeric,maxbibnames=5,giveninits=true,hyperref=true,isbn=false,url=false,doi=true,date=year]{biblatex}

\usepackage{booktabs}

\usepackage[binary-units=true]{siunitx}

\usepackage{algorithm}
\usepackage{algorithmic}
\floatname{algorithm}{Procedure}

\usepackage{graphicx}
\usepackage{hyperref}
\usepackage{bookmark}

\usepackage{enumitem}
\usepackage{tikz}
\usepackage{pgfplots}
\usepackage{pgfplotstable}
\pgfplotsset{compat=newest} 
\pgfplotsset{plot coordinates/math parser=false}
\usepackage{xcolor}
\usepackage{qtree}
\usepackage{hyphenat}
\usepackage[misc]{ifsym}
\usepackage{mathtools}

\KOMAoptions{DIV=last}
\theoremstyle{plain}
\newtheorem{theorem}{Theorem}[section] 
\newtheorem*{theorem*}{Theorem}
\newtheorem{lemma}[theorem]{Lemma} 
\newtheorem*{lemma*}{Lemma} 
\newtheorem{corollary}[theorem]{Corollary} 
\newtheorem*{corollary*}{Corollary} 
 
\newtheorem*{proposition*}{Proposition}

\theoremstyle{definition}
\newtheorem{definition}[theorem]{Definition}
\newtheorem*{definition*}{Definition} 

\theoremstyle{remark} 
\newtheorem{remark}[theorem]{Remark}
\newtheorem*{remark*}{Remark}

\hypersetup{
pdftitle={A parameter-dependent smoother for the multigrid method},
pdfauthor={Lars Grasedyck, Maren Klever, Christian L\"obbert, Tim A. Werthmann}, 
pdfkeywords={multigrid, PDEs, parameter-dependent problems, low-rank tensor formats, exponential sums}
}
\newcommand{\Id}{\operatorname{Id}}   
\newcommand{\diag}{\operatorname{diag}}   
\newcommand{\I}{\mathcal{I}}   
\newcommand{\N}{\mathbb{N}}  
\newcommand{\R}{\mathbb{R}}  
\newcommand{\HT}{\mathcal{H}\operatorname{-Tucker}} 
\definecolor{rwth-50}{RGB}{142 186 229}
\newcommand{\parameterdependent}{parameter\hyp{}dependent}
\newcommand{\righthandside}{right\hyp{}hand~side}

\newcommand{\lowrank}{low\hyp{}rank}
\newcommand{\Galerkinansatz}{Galerkin\hyp{}ansatz}
\newcommand{\twogrid}{two\hyp{}grid}
\begin{document}
\title{A parameter-dependent smoother for the multigrid method}
\author{Lars Grasedyck\thanks{IGPM, RWTH~Aachen University, Templergraben~55, 52056~Aachen, \href{mailto:lgr@igpm.rwth-aachen.de}{lgr@igpm.rwth-aachen.de}} 
\and Maren Klever\thanks{Department of Physics, University of Regensburg, Universit\"atsstra{\ss}e 31, 93040 Regensburg, \newline\href{https://orcid.org/0000-0003-2614-163X}{ORCID~iD:~0000-0003-2614-163X}, \href{mailto:klever@igpm.rwth-aachen.de}{klever@igpm.rwth-aachen.de}} 
\and Christian L{\"o}bbert\thanks{IGPM, RWTH~Aachen University, Templergraben~55, 52056~Aachen, \href{https://orcid.org/0000-0001-8662-6669}{ORCID iD: 0000-0001-8662-6669}, \href{mailto:loebbert@igpm.rwth-aachen.de}{loebbert@igpm.rwth-aachen.de}} 
\and Tim A. Werthmann\thanks{IGPM, RWTH~Aachen University, Templergraben~55, 52056~Aachen, \href{https://orcid.org/0000-0001-9870-9787}{ORCID~iD: 0000-0001-9870-9787}, \href{mailto:werthmann@igpm.rwth-aachen.de}{werthmann@igpm.rwth-aachen.de}}
}
\maketitle
\begin{abstract}
The solution of parameter-dependent linear systems, by classical methods, leads to an arithmetic effort that grows exponentially in the number of parameters.
This renders the multigrid method, which has a well understood convergence theory, infeasible.
A parameter-dependent representation, e.g., a low-rank tensor format, can avoid this exponential dependence, but in these it is unknown how to calculate the inverse directly within the representation.
The combination of these representations with the multigrid method requires a parameter-dependent version of the classical multigrid theory and a parameter-dependent representation of the linear system, the smoother, the prolongation and the restriction.
A derived parameter-dependent version of the smoothing property, fulfilled by parameter-dependent versions of the Richardson and Jacobi methods, together with the approximation property prove the convergence of the multigrid method for arbitrary parameter-dependent representations.
For a model problem low-rank tensor formats represent the parameter-dependent linear system, prolongation and restriction.
The smoother, a damped Jacobi method, is directly approximated in the low-rank tensor format by using exponential sums.
Proving the smoothing property for this approximation guarantees the convergence of the parameter-dependent method.
Numerical experiments for the parameter-dependent model problem, with bounded parameter value range, indicate a grid size independent convergence rate.
\par\vspace\baselineskip\noindent
\par\vspace\baselineskip\noindent
\end{abstract}
\section{Introduction}\label{sec:Introduction}
The modeling of modern scientific problems often leads to partial differential equations. 
Oftentimes one would like to consider a dependence on parameters or uncertainties within a model. 
After suitable discretization one obtains a \parameterdependent\ linear system of equations of the form
\begin{equation*}
A(p) ~ u(p) = f(p),
\end{equation*}
where the operator $A$, the solution $u$ and the \righthandside\ $f$ all depend on the parameter $p \coloneqq (p^{(1)}, p^{(2)}, \dots, p^{(d)})$.

Assuming $n$ different choices for every $p^{(\nu)}$ with $\nu \in \lbrace 1,2, \dots, d \rbrace$, one has to solve $n^d$ linear systems.
This exponential scaling in the dimension $d$ is commonly known as \emph{curse of dimensionality}, which renders classical methods for $d>2$ infeasible.

To overcome the curse of dimensionality, one needs an efficient solver for linear systems on the one hand and on the other hand a tool to exploit the underlying structure of the problem in order to avoid exponential dependency on the number of parameters.
The multigrid method is a state-of-the-art solver for large scale linear systems, since it often scales linearly in the problem size. 
Its convergence theory is well studied and understood, see, e.g.,~\cite{Brandt1977,Hackbusch1985,Xu1992,Yserentant1986}.

For \parameterdependent\ problems multigrid methods have already been successfully used in~\cite{Ballani2013,Grasedyck2017}. 
There the authors modeled the parameter dependency by means of tensor formats~\cite{DeLathauwer2000,Grasedyck2013,Hackbusch2012,Kolda2009,Oseledets2011}, which opens up new possibilities to represent \parameterdependent\ linear systems. 
As these formats allow us to perform arithmetic operations within them, we use tensor formats to formulate a multigrid method for \parameterdependent\ problems. 

While there are results about the solution of \parameterdependent\ linear systems in tensor formats, e.g., based on iterative solvers \cite{Ballani2013,Khoromskij2011,Kressner2011} or local minimization \cite{Dolgov2014,Oseledets2012}, the multigrid method for tensor formats was used in~\cite{Ballani2013,Grasedyck2017} and only discussed shortly in~\cite{Hackbusch2015}.

To the best of our knowledge, there have been no authors who have presented and proven the multigrid theory for \parameterdependent\ problems in a general and detailed setting. 
For the \parameterdependent\ multigrid method, by means of tensor formats, we need an efficient smoother, e.g., the Jacobi method. 
As mentioned by Hackbusch, ``in the tensor case, the performance of this iteration is already too complicated'' and ``[i]nstead one may try to use an approximation''~\cite{Hackbusch2015}.

To solve these problems, in Section~\ref{sec:parameterProblems} we introduce our model problem of a \parameterdependent\ partial differential equation.
In Section~\ref{sec:parameterDepMult} we establish and prove the theory for the \parameterdependent\ multigrid method for a general \parameterdependent\ representation using classical multigrid convergence results.
Further in Section~\ref{sec:repParameterProblems} we use tensor formats to represent the \parameterdependent\ multigrid method. 
For our model problem, we derive a \parameterdependent\ representation of the operator, which we discretize by finite differences.
As a smoother, we approximate the \parameterdependent\ Jacobi method by exponential sums and prove that this approximation fulfills the smoothing property.
We conclude in Section~\ref{sec:NumExp} with numerical experiments, observing a grid size independent convergence behavior of the multigrid method using our \parameterdependent\ Jacobi smoother.
\section{Random diffusion model problem}\label{sec:parameterProblems}
We consider the following PDE as model problem, where the diffusion $\sigma(x,p)$ depends on some parameter $p$:
\begin{equation}\label{eq:cookie} 
\begin{aligned}
-\nabla \cdot \left( \sigma(x, p) \nabla u(x, p) \right) & = f(x) \quad && \text{ in } \Omega,\\
u(x, p) &  = 0 \quad && \text{ on } \partial \Omega.
\end{aligned}
\end{equation}
We assume that $\sigma(x, p)$ is piecewise constant on each $\Omega_{\nu} \subseteq \Omega$ for a partition $(\Omega_{\nu})_{\nu \in \{1 \dots, d\}}$ of $\Omega$. 
Using the multilinearity of the scalar product and the weak formulation of equation~\eqref{eq:cookie}, we can formulate the problem as
\begin{align}\label{eq:cookie2}
	\underbrace{ \left( \sum\limits_{\nu = 1}^d \sigma_{\nu}(p^{(\nu)}) A^{(\nu)} \right) }_{\eqqcolon A(p)} u(p)  = f  ~,
\end{align}
where each $A^{(\nu)}$ only depends on $\Omega_{\nu}$ for $\nu \in \lbrace 1,2, \dots, d \rbrace$, i.e., $A^{(\nu)}$ is parameter-independent.

A similar structure like in equation~\eqref{eq:cookie2} is obtained, e.g., for a general random field with known mean field and covariance, instead of a parameter-dependent diffusion coefficient.
After truncation to a finite number of terms, the Karhunen-Lo{\`e}ve expansion, which separates stochastic and deterministic variables, gives an affine parameter-dependent linear system, see, e.g.,~\cite{Matthies2005,Matthies2012,Schwab2006}.

Since the argumentation in Section~\ref{sec:parameterDepMult} is independent of the underlying representation, the results hold also for other choices of representations, e.g., one could approximate $\sigma$ or the operator, if the diffusion $\sigma$ has a nonlinear dependency, cf.~\cite{Khoromskij2010}.

Here we discretize the parameter space by choosing a finite number of parameters $p \coloneqq (p^{(1)}, \dots, p^{(\nu)}, \dots, p^{(d)})$ from a discrete set $\I$. 
Therefore we need the following definition:
\begin{definition}[Mode and dimension]\label{def:modesDimension}
Let $\I \coloneqq \bigtimes_{\nu=1}^d \I_{\nu}$ be an index set with $\lvert \I_{\nu} \rvert = n_{\nu}$ for all $\nu \in D \coloneqq \lbrace 1, \dots, d \rbrace$. 
We call each $\nu \in D$ \emph{mode} and $d$ the \emph{dimension}. 
\end{definition}
Choosing fixed discrete values for all $p^{(\nu)}$, i.e., $p^{(\nu)} \in \lbrace p^{(\nu)}(1), p^{(\nu)}(2), \dots, p^{(\nu)}(n_{\nu}) \rbrace \eqqcolon \I_{\nu}$, we reformulate the problem as:
\begin{align}\label{eq:problem}
\text{Solve } \quad A(p) u(p) =  f \quad \text{ for all } p \in \I .
\end{align}
Applying classical methods one needs to solve a system of $\prod_{\nu=1}^d n_{\nu} \approx n^d$ linear equations. 
Because of this exponential scaling in $d$, we want to solve the \parameterdependent\ system simultaneously for all $p \in \mathcal{I}$. 
For this reason we need an efficient iterative solution method. 
Next, we recapitulate some convergence results of the classical multigrid method.
\section{Parameter-dependent multigrid method}\label{sec:parameterDepMult}
The complexity of the multigrid method often scales linearly or quasilinearly in the grid size, which makes it well suited for large systems of equations. 
The basic idea is to find a smooth approximation for the error of a given estimate of the fine grid solution on a coarser grid. 
Since the linear system on the coarser grid is smaller, its solution can be computed with less work.
Using this concept again to solve the equation system on the coarser grid, the recursion yields the multigrid method. 
We give the pseudocode of the V-cycle multigrid method in Procedure~\ref{alg:multigrid}.

\begin{algorithm}
\caption{$u_{\ell} \leftarrow \operatorname{multigrid}(u_{\ell}, f_{\ell}, \ell)$ }\label{alg:multigrid}
\begin{algorithmic}
\IF{ $\ell = 0$ }
\STATE{ $ u_{\ell} \leftarrow A_{\ell}^{-1} f_{\ell} $ }
\ELSE 
\STATE{ $u_{\ell} \leftarrow S_{\ell}^{\nu_1}(u_{\ell}, f_{\ell})$  }
\STATE{ $d_{\ell-1} \leftarrow R_{\ell}(f_{\ell} - A_{\ell} u_{\ell}) $ }
\STATE{ $e_{\ell-1} \leftarrow 0 $ }
\STATE{ $e_{\ell-1} \leftarrow \operatorname{multigrid}(e_{\ell-1}, d_{\ell-1}, \ell-1)$ }
\STATE{ $u_{\ell} \leftarrow u_{\ell} + P_{\ell} e_{\ell-1} $ }
\STATE{ $u_{\ell} \leftarrow S_{\ell}^{\nu_2}(u_{\ell}, f_{\ell})$  }
\ENDIF
\end{algorithmic} 
\end{algorithm}

Based on the convergence theory of Hackbusch~\cite{Hackbusch2016,Hackbusch1985}, we want to analyze the convergence of this method for \parameterdependent\ problems. 
Therefore we define the following notation:

Let $X_{\ell}$ denote the grid of level $\ell \in \N$ with grid size $h_{\ell} > 0$ and let $A_{\ell}$ denote the matrix corresponding to this grid. 
Analogously let us denote with $X_{\ell-1}$ the next coarser grid with corresponding matrix $A_{\ell-1}$. 
Further let $P_{\ell} : X_{\ell-1} \rightarrow X_{\ell}$ denote the prolongation and $R_{\ell} : X_{\ell} \rightarrow X_{\ell-1}$ the restriction, as well as $S_{\ell}$ the iteration matrix of the smoother corresponding to the grid of level $\ell$. 
We define the iteration matrix of the \twogrid\ method including $\nu_1 \in \N$ presmoothing steps and $\nu_2 \in \N_{0}$ postsmoothing steps by
	\begin{align*}
		M_{\ell}^{\operatorname{TGM}(\nu_1,\nu_2)} \coloneqq S_{\ell}^{\nu_2} \left( \Id - P_{\ell} A_{\ell-1}^{-1} R_{\ell} A_{\ell}\right) S_{\ell}^{\nu_1} ~.
	\end{align*}
For convenience we will choose $\nu_1 = \nu$ and $\nu_2 = 0$ in the following theoretical analysis of the method.
We will further denote the \parameterdependent\ version of an arbitrary object $G$ by $G(p)$, e.g., $A_{\ell}(p)$ denotes the \parameterdependent\ operator, which corresponds to the grid of level $\ell$.
We also assume that $A_{\ell}(p)$ is symmetric positive definite for all $p\in \mathcal{I}$.
	
The main idea to prove the convergence of the multigrid method is to split the iteration matrix of the method into two parts
	\begin{align*}
		\lVert M_{\ell}^{\operatorname{TGM}(\nu,0)}  \rVert_2 \leq \lVert A_{\ell}^{-1} - P_{\ell} A_{\ell-1}^{-1} R_{\ell} \rVert_2 ~ \lVert A_{\ell} S_{\ell}^{\nu} \rVert_2 ~,
	\end{align*}
where the first part must fulfill the \emph{approximation property} and the second part must fulfill the \emph{smoothing property}. 
Next, we extend these properties to the \parameterdependent\ case.
	\subsection{Smoothing property}\label{sec: Theory about the smoothing property for $A(p)$}
We define the smoothing property in the \parameterdependent\ case analogously to~\cite[Definition $11.25$]{Hackbusch2016}.
\begin{definition}[Smoothing property]\label{def:smoothingPropertyParameter}
For an iteration with \parameterdependent\ iteration matrix $S_{\ell}(p)$ the \emph{smoothing property} is defined as
\begin{align*}
\lVert A_{\ell}(p) S_{\ell}^{\nu}(p) \rVert_2 \leq \eta(\nu) \lVert A_{\ell}(p) \rVert_2
\end{align*}
for all $0 \leq \nu \leq \overline{\nu}(h_{\ell})$, $p \in \I$ and ${\ell} \in \N_{0}$, with functions $\eta$ and $\overline{\nu}$ satisfying 
\begin{align*}
&\lim\limits_{\nu \rightarrow \infty} \eta(\nu) = 0,\\
&\lim\limits_{h \rightarrow 0} \overline{\nu}(h) = \infty \quad \text{ or } \quad \overline{\nu}(h) = \infty
\end{align*}
independent of the level $\ell$.
\end{definition}
This property ensures that the approximation of the fine grid error is smooth enough to be approximated on the coarser grid.
We now define the iteration matrix of the damped Richardson method in case of \parameterdependent\ problems.
\begin{definition}[Damped Richardson method]\label{def:dampedRichardsonMethod}
The \emph{damped Richardson method} is
\begin{align*}
u^{j+1}_{\ell}(p) = u^{j}_{\ell}(p) + \omega(p) \left( f_{\ell} - A_{\ell}(p) u^{j}_{\ell}(p) \right),
\end{align*}
with damping factor $\omega(p) \in \left(0,1\right] $, parameter $p \in \I$ and $\ell \in \N_{0}$. 
Its iteration matrix is given by
\begin{align*}
	S_{\text{Rich,}\omega, \ell}(p) \coloneqq \Id - \omega(p) A_{\ell}(p).
\end{align*}
\end{definition}
To define the iteration matrix of the damped Jacobi method, we denote with $D_{\ell}(p) \coloneqq \diag(A_{\ell}(p))$ the diagonal of $A_{\ell}(p)$.
\begin{definition}[Damped Jacobi method]\label{def:dampedJacobiMethodParameter}
The \emph{damped Jacobi method} is
\begin{align*}
u^{j+1}_{\ell}(p) = u^{j}_{\ell}(p) + \omega(p) D_{\ell}^{-1}(p)  \left( f_{\ell} - A_{\ell}(p) u^{j}_{\ell}(p) \right),
\end{align*}
with damping factor $\omega(p) \in \left(0,1\right]$, parameter $p \in \I$ and $\ell \in \N_{0}$.
Its iteration matrix is given by
	\begin{align*}
		S_{\text{Jac,}\omega, \ell}(p) \coloneqq \Id - \omega(p) D_{\ell}^{-1}(p) A_{\ell}(p).
	\end{align*}
\end{definition}
Next, we verify the smoothing property for those methods.
Therefore we formulate and prove the \parameterdependent\ version of a classic result~\cite[Lemma $11.23$]{Hackbusch2016}.
\begin{lemma}\label{lemma:smoothingPropertyParameter}
Let $0 \leq B(p) = B^{T}(p) \leq \Id$ for all $p\in \mathcal{I}$, then for any $\nu \in \N_0$
\begin{align*}
	\left\lVert B(p) \left( \Id - B(p)  \right)^{\nu} \right\rVert_2 \leq \eta_{0}( \nu )
\end{align*} 
holds for all $p \in \mathcal{I}$ with $\eta_{0}( \nu ) \coloneqq \frac{\nu^{\nu}}{( \nu + 1 )^{\nu + 1}} $.
\end{lemma}
\begin{proof}
We follow the proof from~\cite[Lemma~$11.23$]{Hackbusch2016} using a \parameterdependent\ representation.

For fixed $p \in \I$ the matrix $B(p) ( \Id - B(p)  )^{\nu}$ is symmetric. 
Let $\lambda(p) \in \mathbb{R}_{+}$ denote an arbitrary eigenvalue of $B(p)$. 
Then $\lambda(p) ( 1 - \lambda(p) )^{\nu}$ is an eigenvalue of $B(p) ( \Id - B(p)  )^{\nu}$. 
Because $0 \leq B(p) \leq \Id$ holds, $0 \leq \lambda(p) \leq 1 $ and therefore $0 \leq \lambda(p) ( 1 - \lambda(p) )^{\nu} \leq 1$ follows. 
Consequently $B(p) ( \Id - B(p)  )^{\nu}$ is positive semi-definite with
\begin{align*}
\sup\limits_{p \in \I} \left\lVert B(p) \left( \Id - B(p)  \right)^{\nu} \right\rVert_2 & = \max\limits_{p \in \I} \left( \max\limits_{\lambda(p) \in  \sigma(B(p) )} \lambda(p) \left( 1 - \lambda(p) \right)^{\nu} \right) \\
&\leq \max\limits_{ \lambda \in \left[0, 1\right] }  \lambda \left( 1 - \lambda \right)^{\nu} = \eta_{0}(\nu).
\end{align*}
The last equality follows by maximizing the term and hence the lemma follows.
\end{proof}
As $\lim_{\nu \to \infty} \eta_{0}( \nu ) =0$ holds and $ \eta_{0}( \nu ) \leq \frac{1}{2 (\nu+1)}$ for $\nu>1$ holds, we can prove the smoothing property for the damped Richardson method and the damped Jacobi method in case of symmetric positive definite operators.
\begin{theorem}\label{th:smoothingRichardsonParameter}
Let $\omega(p) = \frac{c_0}{\rho(A_{\ell}(p))}$ for a constant $c_0 \in (0, 1]$, then
	\begin{align*}
		\left\lVert A_{\ell}(p) S_{\text{Rich,}\omega, \ell}^{\nu}(p) \right\rVert_2 \leq \frac{1}{2 c_0 \left(\nu + 1 \right)} \Vert A_{\ell}(p) \Vert_2  \quad \forall ~ \nu \in \N 
	\end{align*}
	holds for all $p \in \I$ and $\ell \in \N_{0}$.
\end{theorem}
\begin{proof}
With Definition~\ref{def:dampedRichardsonMethod} and Lemma~\ref{lemma:smoothingPropertyParameter} we calculate
\begin{equation*}
	\begin{aligned}
	& \left\lVert  A_{\ell}(p) S_{\text{Rich,}\omega, \ell}^{\nu}(p) \right\rVert_2 = \frac 1{\omega(p)} \left\lVert \omega(p) A_{\ell}(p) S_{\text{Rich,}\omega, \ell}^{\nu}(p) \right\rVert_2 \\
	= & \frac 1{\omega(p)} \left\lVert \omega(p) A_{\ell}(p)  \left(\Id - \omega(p) A_{\ell}(p)\right)^{\nu} \right\rVert_2 \leq  \frac{1}{2 \omega(p) ( \nu + 1 )} \\
\leq &  \frac{1}{2 c_0 \left( \nu + 1 \right)} \rho( A_{\ell}(p) )\leq \frac{1}{2 c_0 \left( \nu + 1 \right)} \left\lVert A_{\ell}(p) \right\rVert_2.
	\end{aligned}
\end{equation*}
As $\rho ( S_{\text{Rich,}\omega, \ell}(p) ) < 1$ holds for all $p \in \mathcal{I}$, the method converges and hence $\overline{\nu}(h)= \infty$ follows. 
\end{proof}
Therefore the damped Richardson method~\ref{def:dampedRichardsonMethod} fulfills the smoothing property.
\begin{theorem}\label{th:smoothingJacobiParameter}
Let $\omega(p) \in (0, \rho( D_{\ell}^{-1}(p) A_{\ell}(p) )^{-1} ]$, then
	\begin{align*}
		\left\lVert A_{\ell}(p) S_{\text{Jac,}\omega, \ell}^{\nu}(p) \right\rVert_2 \leq \frac{1}{ 2 \omega(p) \left(\nu + 1 \right)} ~ \lVert A_{\ell}(p) \rVert_2 \quad \forall ~ \nu \in \N 
	\end{align*}
	holds for all $p \in \I$ and $\ell \in \N_{0}$.
\end{theorem}
\begin{proof}
As $A_{\ell}(p)$ is symmetric positive definite for all $p \in \mathcal{I}$, its diagonal satisfies $D_{\ell}(p) > 0$. 
Therefore we can define $\tilde{A}_{\ell}(p) \coloneqq D_{\ell}^{-\frac{1}{2}}(p) A_{\ell}(p) D_{\ell}^{-\frac{1}{2}}(p)$.
As now $0 \leq \omega(p) \tilde{A}_{\ell}(p) \leq \Id$ holds, together with Definition~\ref{def:dampedJacobiMethodParameter} and Lemma~\ref{lemma:smoothingPropertyParameter}, we calculate
	\begin{align*}
&\lVert A_{\ell}(p) S_{\text{Jac,}\omega, \ell}^{\nu}(p)\rVert_2 = \lVert A_{\ell}(p) ( \Id - \omega(p) D_{\ell}^{-1}(p) A_{\ell}(p) )^{\nu} \rVert_2 \\
		  \leq & \frac{1}{\omega(p)} \lVert D_{\ell}^{\frac{1}{2}}(p) \rVert_2^2 \lVert \omega(p) \tilde{A}_{\ell}(p) \left( \Id - \omega(p) \tilde{A}_{\ell}(p) \right)^{\nu} \rVert_2\\
		 \leq & \frac{1}{2 \omega(p) \left(\nu + 1 \right) } \left\lVert D_{\ell}(p) \right\rVert_2 \leq \frac{1}{2 \omega(p) \left(\nu + 1 \right) } \left\lVert A_{\ell}(p) \right\rVert_2	~.
	\end{align*}
As $\rho ( S_{\text{Jac,}\omega, \ell}(p) ) < 1$ holds for all $p \in \mathcal{I}$, the method converges and hence $\overline{\nu}(h)= \infty$ follows. 
\end{proof}
Therefore the damped Jacobi method~\ref{def:dampedJacobiMethodParameter} fulfills the smoothing property.

We remark that one could choose $\omega = \min_{p \in \I} \omega(p)$ to obtain a damping factor independent of $p$ in Theorems~\ref{th:smoothingRichardsonParameter} and~\ref{th:smoothingJacobiParameter}.
Although this choice guarantees the convergence of both methods for all $p \in \I$, it can be suboptimal for most $p$, e.g., if the value of $\omega(p)$ varies by orders of magnitude.
This complicates the choice of a uniform optimal damping factor for the Richardson method and is a disadvantage compared to the Jacobi method, where the multiplication of the operator with its inverse diagonal scales the value range of $\omega(p)$.
	\subsection{Approximation property}
We define the approximation property for \parameterdependent\ problems analogously to the classic definition in~\cite[$11.6.3.1$]{Hackbusch2016}.
\begin{definition}[Approximation property]\label{def:approximationPropertyParameter}
The \emph{approximation property} is given by
\begin{align*}
\left\lVert A_{\ell}^{-1}(p) - P_{\ell}(p) A_{\ell-1}^{-1}(p) R_{\ell}(p) \right\rVert_2 \leq \frac{C_{A}}{ \left\lVert A_{\ell}(p) \right\rVert_2 } \quad \forall ~ \ell \in \N
\end{align*}
with a constant $C_A > 0$ independent of $\ell$.
\end{definition}
The approximation property assures that the error on the coarse grid is a good approximation of the error on the fine grid.
Using Definition~\ref{def:approximationPropertyParameter} to verify the approximation property involves the calculation of the \parameterdependent\ inverse. 
The computational effort makes this infeasible for $d>2$.
One would thus like to have a theorem, which allows the proof of the approximation property in the \parameterdependent\ case, involving only the operators themselves and not their inverses. 

As a first idea, we propose the following ansatz by Hackbusch~\cite[$11.6.3.1$]{Hackbusch2016}.
Let the \Galerkinansatz\
\begin{align*}
	A_{\ell-1}(p) = R_{\ell}(p) A_{\ell}(p) P_{\ell}(p)
\end{align*}
hold for all $p\in \I$, $\ell \in \N$ and let further $p \in \mathcal{I}$ fixed but arbitrary. 
For an arbitrary restriction $R'(p) : X_{\ell} \to X_{\ell-1}$ the following factorization holds
	\begin{align*}
		A_{\ell}^{-1}(p) - P_{\ell}(p)  A_{\ell-1}^{-1}(p)  R_{\ell}(p) =& \left[ \Id - P_{\ell}(p)  A_{\ell-1}^{-1}(p)  R_{\ell}(p)  A_{\ell}(p)  \right] A_{\ell}^{-1}(p)  \\
		= &\left[ \Id - P_{\ell}(p)  A_{\ell-1}^{-1}(p)  R_{\ell}(p)  A_{\ell}(p)  \right] \left[ \Id - P_{\ell}(p) R'(p) \right] A_{\ell}^{-1}(p)  .
	\end{align*}
If the solution $u_{\ell}(p) \coloneqq A_{\ell}^{-1}(p)  f_{\ell}$ is sufficiently smooth, e.g., discrete regular, the interpolation error
\begin{align*}
		d_{\ell}(p)  = \left[ \Id - P_{\ell}(p) R'(p) \right] u_{\ell}(p)  = u_{\ell}(p)  - P_{\ell}(p) R'(p)  u_{\ell}(p) 
\end{align*}
can be estimated by $\lVert d_{\ell}(p)  \rVert_2 \leq C \frac{\lVert f_{\ell} \rVert_2}{\lVert A_{\ell}(p) \rVert_2}$ with $C > 0$ independent of $\ell$. 
The same argument can be used to show
\begin{align*}
		\left\lVert \Id - P_{\ell}(p)  A_{\ell-1}^{-1}(p)  R_{\ell}(p)  A_{\ell}(p)  \right\rVert_2 \leq C,
\end{align*}
which implies the approximation property. 
Following this idea, using the \Galerkinansatz, one can prove the next result under some additional requirements.
\begin{theorem}\label{th:galerkinAnsatz}
Let the \emph{\Galerkinansatz}
\begin{align*}
	A_{\ell-1}(p) = R_{\ell}(p) A_{\ell}(p) P_{\ell}(p)
\end{align*}
hold for all $p\in \I$ and $\ell \in \N$, where one chooses $R_{\ell}(p)$ and $P_{\ell}(p)$ as the canonical restriction and prolongation for all $p\in \I$ and $\ell \in \N$. 
We assume that for $m\in \N$ with constants $C_E$, $C_K$, $C_P$ and $C_h > 0$ independent of $\ell$
\begin{equation*}
\begin{aligned}
\lVert A^{-1}(p) - P_{\ell}(p) A_{\ell}^{-1}(p) R_{\ell}(p) \rVert & \leq C_{E} h^m_{\ell}~,\\
\lVert  A_{\ell}(p) \rVert_2 & \leq C_{K} h^{-2m}_{\ell} ~,\\
\lVert  P_{\ell}(p) \left(R_{\ell}(p) P_{\ell}(p) \right)^{-1} \rVert & \leq C_{P}~,\\
\lVert \left(R_{\ell}(p) P_{\ell}(p) \right)^{-1}  R_{\ell}(p)  \rVert & \leq C_{P} ~,\\
\text{and } h_{\ell-1} & \leq C_{h} h_{\ell}
\end{aligned}
\end{equation*}
hold, then the approximation property is satisfied.
\end{theorem}
\begin{proof}
For fixed $p\in \I$ the theorem follows from~\cite[Theorem $11.34$]{Hackbusch2016}.
Because we choose $p \in \I$ arbitrary, the theorem also holds in the \parameterdependent\ case.
\end{proof}
	\subsection{Convergence}
By proving the smoothing and approximation properties one obtains the convergence of the \twogrid\ method, presented in the following theorem.
\begin{theorem}\label{th:convergenceTwoGrid}
Assume the smoothing property of Definition~\ref{def:smoothingPropertyParameter} with $\overline{\nu}(h)= \infty$ and the approximation property of Definition~\ref{def:approximationPropertyParameter} are fulfilled. 
For a given $0 < \zeta < 1$, there exists a lower bound $\underline{\nu} \in \N_0$, such that
	\begin{align*}
		\left\lVert M_{\ell}^{\operatorname{TGM}(\nu,0)}(p) \right\rVert_2 \leq C_{A} \eta(\nu) \leq \zeta
	\end{align*}
holds for all $\nu \geq \underline{\nu}$, $\ell \in \N$ and $p \in \mathcal{I}$.
\end{theorem}
\begin{proof}
We can factorize the \twogrid\ iteration matrix via
\begin{align*}
	M_{\ell}^{\operatorname{TGM}(\nu,0)}(p) = & \left[ \Id- P_{\ell}(p) A^{-1}_{\ell-1}(p) R_{\ell}(p) A_{\ell}(p) \right] S_{\ell}^{\nu}(p)\\
	= & \left[ A^{-1}_{\ell}(p) - P_{\ell}(p) A^{-1}_{\ell-1}(p) R_{\ell}(p) \right] \left[ A_{\ell}(p) S_{\ell}^{\nu}(p) \right].
\end{align*}
Using the smoothing property and the approximation property we directly obtain the result. 
\end{proof}
We have thus proved the convergence of the \twogrid\ method for \parameterdependent\ problems. 
Using the convergence of the \twogrid\ method, one obtains the multigrid convergence with help of some weak additional assumptions, similarly to the classical case, cf., e.g.,~\cite[Theorem $11.42$]{Hackbusch2016}.
\section{Representation of \parameterdependent\ problems}\label{sec:repParameterProblems}
Because of our \parameterdependent\ multigrid theory we now introduce representations of the operator, the solution, the \righthandside, the smoother, the prolongation and the restriction in a \parameterdependent\ way, such that we can perform arithmetic operations with them.

One possible way for the representation of \parameterdependent\ problems are \lowrank\ tensor formats, cf., e.g.,~\cite{Ballani2013,Grasedyck2017,Kressner2011}.
To illustrate the idea, we assume that the parameter dependency is a scaling of a given operator $A$, i.e., $\sigma(\nu) A$ and that the \righthandside\ $f$ is constant for all $\sigma(\nu)$.
With classical methods we would have to solve the following linear system.
\begin{align*}
	\begin{pmatrix}
			\sigma(1) A & 0 & \hdots & 0 \\
			0 & \sigma(2) A & \ddots & \vdots \\
			\vdots & \ddots & \ddots & 0 \\
			0 & \hdots  & 0 & \sigma(n) A
	\end{pmatrix}
	\begin{pmatrix}
			u(\sigma(1)) \\
			u(\sigma(2)) \\
			\vdots \\
			u( \sigma(n))
	\end{pmatrix}
	=
	\begin{pmatrix}
			f \\
			f \\
			\vdots \\
			f
	\end{pmatrix}.
\end{align*} 
If we model this system using the Kronecker product
\begin{align*}
	\left( \diag(\sigma(1), \sigma(2), \dots, \sigma(n)) \otimes A \right) u(\sigma) = \left(1, \cdots, 1\right)^T \otimes f ~,
\end{align*}
we derive a data-sparse representation. 
We now generalize the above representation for the case of more than one parameter.
	\subsection{Operator}
For a one-dimensional geometry, equation~\eqref{eq:cookie} reads
\begin{equation} \label{eq:cookie1D} 
	\begin{aligned}
	-\frac{\partial}{\partial x} \left( \sigma(x,p) \frac{\partial}{\partial x} u(x,p) \right) & = f(x) \quad &&\text{ in } \Omega, \\
	u(x,p) & = 0 \quad &&\text{ on } \partial\Omega.
	\end{aligned}
\end{equation}
Let $n \in \N$.
We denote the grid size by $h \coloneqq \frac 1n > 0$, the grid points by $x_i$ with $i \in \lbrace 0,\dots, n \rbrace$, and the discrete diffusion at grid point $x_i$ by $\sigma_{i}$.
For ease of presentation, we consider only finite difference discretization and therefore assume that $u(x,p) \in C^4$ and $\sigma(x,p) \in C^1$ for $x \in \Omega$. 
A possible future work could be the generalization of the theoretical results, e.g., to the finite element method, where one has weaker requirements of regularity.
\begin{theorem}\label{th:cookie1DStencil}
For equation~\eqref{eq:cookie1D} a second-order consistent stencil is given by
\begin{align*}
	\frac{1}{h^2}
		\begin{bmatrix}
			- \frac{\sigma_{i-1} + \sigma_{i}}{2} & \frac{ \sigma_{i-1} + 2 \sigma_{i} + \sigma_{i+1} }{2} & -\frac{ 	\sigma_{i} + \sigma_{i+1} }{2}
		\end{bmatrix}.
\end{align*}
\end{theorem}
\begin{proof}
Taylor's theorem and equating the coefficients of
\begin{align*}
(Au)_i & = - \left( \sigma'(x_i,p) u'_i + \sigma(x_i,p) u''_i \right) \quad \text{ and, }\\
(A_h u_h)_i & = - \frac{1}{h^2} \left( -\tilde{\sigma}_{i} u_{i-1} + (\tilde{\sigma}_{i} + \tilde{\sigma}_{i+1}) u_{i} - \tilde{\sigma}_{i+1} u_{i+1} \right),
\end{align*}
yields $\tilde{\sigma}_i = \frac{\sigma_{i-1}+\sigma_{i}}{2}$ for a second-order consistent stencil.
\end{proof}
Using this result, we derive an affine representation of the discrete operator.
\begin{corollary}\label{corol:aff1D}
An affine representation of the discrete operator in the one-dimensional case is given by
\begin{align*}
	\frac{1}{h^2} \left(
	\sigma_{i-1}
	\begin{bmatrix}
	- \frac{1}{2} & \frac{1}{2} & 0
	\end{bmatrix}
	+ \sigma_{i}
	\begin{bmatrix}
	- \frac{1}{2} & 1 & -\frac{1}{2}
	\end{bmatrix}
	+ \sigma_{i+1}
	\begin{bmatrix}
	0 & \frac{ 1 }{2} & -\frac{ 1}{2}
	\end{bmatrix}
	\right) .
\end{align*}
\end{corollary} 
\begin{proof}
Follows from Theorem~\ref{th:cookie1DStencil} with linearity.
\end{proof}
We notice that in case of $\sigma_i = 1$ for all $i \in \lbrace 0,\dots, n \rbrace$ the discretization from Corollary~\ref{corol:aff1D} is equivalent to the standard discretization of the Laplace operator and that the diagonal of a local operator $A^{(\mu)}$ with $\mu \in \{1, \dots, d\}$ is given by 
\begin{align*}
	\diag(A^{(\mu)}) = \frac{1}{2h^2} \diag\left(0,\dots,0,1, 3, 4,\dots, 4,3, 1, 0, \dots, 0\right).
\end{align*}

Further we generalize these results to a two-dimensional geometry, there equation~\eqref{eq:cookie} reads
\begin{equation}\label{eq:cookie2D} 
	\begin{aligned}
		 -\nabla \cdot \left( \sigma(x,y,p) \nabla u(x,y,p) \right)  &= f(x,y) \quad &&\text{ in } \Omega, \\
		u(x,y,p) & = 0 \quad &&\text{ on } \partial\Omega,
	\end{aligned}
\end{equation}
where the left-hand side of the PDE is equal to 
\begin{equation*} 
	\begin{aligned}
		& -\nabla \cdot \left( \sigma(x,y,p) \nabla u(x,y,p) \right) \\  
		 = & -\frac{\partial}{\partial x} \left( \sigma(x,y,p) \frac{\partial}{\partial x} u(x,y,p) \right)- \frac{\partial}{\partial y} \left( \sigma(x,y,p) \frac{\partial}{\partial y} u(x,y,p) \right).
	\end{aligned}
\end{equation*}
\begin{theorem}\label{th:cookie2DStencil}
For equation~\eqref{eq:cookie2D} a second-order consistent stencil is given by
\begin{align*}
\frac{1}{h^2}
\begin{bmatrix}
0 & - \frac{\sigma_{i,i-1} + \sigma_{i,i}}{2} & 0 \\
- \frac{\sigma_{i-1,i} + \sigma_{i,i}}{2} & \frac{ \sigma_{i-1,i} + \sigma_{i,i-1} + 4 \sigma_{i,i} + \sigma_{i,i+1} + \sigma_{i+1,i} }{2} & -\frac{ \sigma_{i,i} + \sigma_{i+1,i} }{2} \\
0 &  -\frac{ \sigma_{i,i} + \sigma_{i,i+1} }{2} & 0
\end{bmatrix},
\end{align*}
where $\sigma_{i,j}$ denotes the discrete evaluation of $\sigma(x,y,p)$ at $(x_i,y_j)$.
\end{theorem}
\begin{proof}
The statement follows from Theorem~\ref{th:cookie1DStencil} taking into account the structure of equation~\eqref{eq:cookie2D}. 
\end{proof}
Next, we derive an affine representation of the discrete operator.
\begin{corollary}\label{corol:aff2D}
An affine representation of the discrete operator in the two-dimensional case is given by
\begin{equation*}
	\frac{1}{h^2} \left(
	\begin{aligned}
		& \sigma_{i,i-1} ~M_{i,i-1} && \\
		+ ~\sigma_{i-1,i} ~M_{i- 1,i} ~~
		+~& \sigma_{i,i} ~M_{i,i} \;
		&&+ ~\sigma_{i+1,i} ~M_{i+1,i}\\
		+~& \sigma_{i,i+1}~ M_{i,i+1} &&
	\end{aligned}
	\right),
\end{equation*}
where
\begin{align*}
	&M_{i, i-1} =
		\begin{bmatrix}
		0 & - \frac{1}{2} & 0 \\
		0 & \frac{ 1 }{2} & 0 \\
		0 & 0 & 0
		\end{bmatrix},
	&&M_{i-1,i} =
		\begin{bmatrix}
		0 & 0 & 0 \\
		- \frac{1}{2} & \frac{ 1 }{2} & 0 \\
		0 & 0 & 0
		\end{bmatrix},\\
	&  M_{i,i} = 
		\begin{bmatrix}
			0 & - \frac{1}{2} & 0 \\
			- \frac{1}{2} & 2 & -\frac{ 1 }{2} \\
			0 &  -\frac{1 }{2} & 0
		\end{bmatrix}, \\
	&M_{i+1,i} =
	\begin{bmatrix}
	0 & 0 & 0 \\
	0 & \frac{1 }{2} & -\frac{ 1 }{2} \\
	0 &  0 & 0
	\end{bmatrix},
	&& M_{i,i+1} =
	\begin{bmatrix}
	0 & 0 & 0 \\
	0 & \frac{1 }{2} & 0 \\
	0 &  -\frac{ 1 }{2} & 0
	\end{bmatrix}	.
\end{align*}
\end{corollary} 
\begin{proof}
Follows from Theorem~\ref{th:cookie2DStencil} with linearity.
\end{proof}
Therefore the operator of equation~\eqref{eq:cookie}, discretized by the finite-difference method, has an affine structure of the form
\begin{align*}
	A(p) \coloneqq A^{(0)} + \sum_{\nu=1}^d p^{(\nu)} A^{(\nu)} .
\end{align*}
One could also use, e.g., the finite-element method, to derive an affine operator structure~\cite{Grasedyck2017,Kressner2011}.

	For the computation of the solution of equation~\eqref{eq:problem} for all possible $p \in \I$, we could define a large block-diagonal system with the following operator
\begin{align*}
	\mathcal{A} & =
	\begin{pmatrix}
		\mathbf{A}_{1} ^{(0)}& 0 & \hdots & 0 \\
		0 & 	\mathbf{A}_{2} ^{(0)} & \ddots & \vdots \\
		\vdots & \ddots & \ddots & 0 \\
		0 & \hdots  & 0 & 	\mathbf{A}_{n} ^{(0)}
	\end{pmatrix} 
	\eqqcolon \operatorname{blkdiag}\left(\mathbf{A}_{1} ^{(0)}, \dots, \mathbf{A}_{n} ^{(0)} \right)~,
\end{align*}
where $	\mathbf{A}_{\mu} ^{(0)} =  A^{(0)}+ \sum_{\nu = 1}^{d} p^{(\nu)}(\mu) A^{(\nu)} $ denotes the $\mu$th diagonal block.

Now, however, the memory requirement to store $\mathcal{A}$ grows exponentially in $n$ and so, even for moderate values of $d$ and $n_{\nu}$, a classical representation of our problem is infeasible. 
Therefore we want to reformulate the problem.
Using the notation $\mathbf{A}_{\mu} ^{(i)} =  \sum_{\nu = i}^{d} p^{(\nu)}(\mu) A^{(\nu)} $ for an $i \in \{1,\dots, d\}$, we get:
\begin{align*}
 \mathcal{A} = & \operatorname{blkdiag}\left(A^{(0)} + \mathbf{A}_{1} ^{(1)}, A^{(0)}+ \mathbf{A}_{2} ^{(1)}, \dots, A^{(0)}+ \mathbf{A}_{n} ^{(1)} \right) \\
		=&
		\operatorname{blkdiag}\left( A^{(0)}, A^{(0)}, \dots, A^{(0)} \right)\\
		& \quad +
		\operatorname{blkdiag}\left( p^{(1)}(1) A^{(1)} , p^{(1)}(2) A^{(1)} , \dots ,  p^{(1)}(n_1) A^{(1)} \right)\\
		& \quad +
		\operatorname{blkdiag}\left( \mathbf{A}_{1}^{(2)} , \mathbf{A}_{2}^{(2)} , \dots, \mathbf{A}_{n}^{(2)} \right)\\
		= & \Id_{n_{d}} \otimes \dots \otimes \Id_{n_{2}} \otimes \Id_{n_{1}} \otimes A^{(0)} \\
		& \quad +  \Id_{n_{d}} \otimes \dots \otimes \Id_{n_{2}} \otimes \operatorname{diag}\left(p^{(1)} \right) \otimes A^{(1)} \\
		& \quad + \dots  + \operatorname{diag}\left(p^{(d)} \right)  \otimes \dots \otimes \Id_{n_{2}} \otimes \Id_{n_{1}} \otimes A^{(d)}.
\end{align*}
This leads to the following data-sparse form of the operator
\begin{align*}
	\mathcal{A} = \sum\limits_{\nu = 0}^d \bigotimes\limits_{\mu= 0}^d A^{(\nu)}\left( \mu \right),\,
\end{align*}
where
\begin{align*}
	A^{(\nu)}\left( \mu \right) = 
	\begin{cases} 			
	A^{(\nu)}& \text{if } \mu = d,\\
	\operatorname{diag}\left(p^{(\nu)}\right)& \text{if } \mu+ \nu = d  \text{ and } \nu \neq 0, \\
	\Id_{n_{d-\nu}}& \text{otherwise}
	\end{cases}
\end{align*}
with $p^{(\nu)} = \left(p^{(\nu)}(1), \dots, p^{(\nu)}(n_{\nu})\right)$.
Similar results can be obtained for the \righthandside.

Such a representation is called a \emph{CANDECOMP/PARAFAC}, or short CP, \emph{representation}, cf.\ Definition~\ref{def:CPDecomposition}.
		\subsubsection{The CP decomposition}\label{subsec:CPDecomposition}
\begin{definition}[CP decomposition]\label{def:CPDecomposition}
	A \emph{CP representation} of a tensor $\mathcal{B}\in \R^{\I}$ is defined as
	\begin{align}
	\label{eq:CPDecomposition}
		\mathcal{B} = \sum\limits_{\nu = 1}^k \bigotimes\limits_{\mu = 1}^d b_{\nu}^{\left( \mu \right)} \quad \text{ with } b_{\nu}^{\left( \mu \right)} \in \R^{\I_{\mu}},
	\end{align}
where $\I = \bigtimes_{\mu=1}^d \I_{\mu}$ is an index set and $k \in \N_0$ is the according \emph{representation rank}.
The minimal $k$ is called the \emph{CP rank} of $\mathcal{B}$ and in this case equation~\eqref{eq:CPDecomposition} is called the \emph{CP decomposition} of $\mathcal{B}$.
Tensors of the form $\bigotimes_{\mu = 1}^d b_{\nu}^{( \mu )} $, i.e., rank $1$, are called \emph{elementary tensors}.
\end{definition}
A big advantage of the CP format is the data-sparsity in case of a small representation rank $k$, since a tensor $\mathcal{B} \in \R^{\I}$ of the form~\eqref{eq:CPDecomposition} has storage cost in $\mathcal{O}(k \sum_{\mu=1}^d \vert \I_{\mu} \vert )$.

The problem of finding conditions for the existence of a \lowrank\ approximation for a given tensor is a research topic of its own~\cite{Bachmayr2017,Dahmen2016,Kressner2011,Kressner2016}.
But since this goes beyond the scope of this article, we assume that our solution has a \lowrank\ approximation, as we want to solve a \parameterdependent\ linear system using \lowrank\ tensor formats.

For operators with rank $k>1$ no algorithm is known that can calculate the inverse of such an operator in a direct way. 
Hence, we need iterative solvers and arithmetic operations within this formats. 
Such arithmetic operations often lead to a growth of the representation rank and therefore we need a truncation down to smaller rank.
For matrices, the multigrid method together with truncation was used in~\cite{Grasedyck2007} to solve large-scale Sylvester equations. 
Since we want to guarantee the convergence of our iterative method, we have to guarantee that the truncation error is small enough, cf.~\cite{Hackbusch2008}, because then the iterative method will still converge.

The problem concerning the CP decomposition is that the set of CP tensors of (almost) rank $k$ is not closed. 
This makes the approximation of a CP tensor of (almost) rank $k$ an ill-posed problem and therefore we cannot guarantee that the truncation error will be small enough.
Because of this, we use the \emph{hierarchical Tucker} format to represent the solution of a linear system.
		\subsubsection{The hierarchical Tucker decomposition}
Next, we recall the hierarchical Tucker format, which was first introduced in~\cite{Hackbusch2009} and further analyzed in~\cite{Grasedyck2010}.

Our solution $u(p)$ depends on the parameters $p \coloneqq ( p^{(1)}, \dots, p^{(d)} )\in \I$, thus we can interpret the solution $\mathcal{U} \in \mathbb{R}^{\I}$ as a tensor of dimension $d\in \mathbb{N}$, where $\I = \bigtimes_{\mu = 1}^d \I_{\mu}$ is a finite product index set.
We call each $\mu \in \lbrace 1, \dots, d \rbrace$ a \emph{mode}.

The general idea of the hierarchical Tucker format is to define a hierarchy among the modes $D \coloneqq \lbrace 1, \dots, d \rbrace$. 
To do so, we define the so called \emph{dimension tree} $\mathcal{T}$ analogously to~\cite[Definition 3.1]{Grasedyck2010}.
\begin{definition}[Dimension tree]\label{def:dimensionTree}
	A \emph{dimension tree} $\mathcal{T}$ for dimension $d \in \N$ is a binary tree with nodes labeled by non-empty subsets of $D$. 
Its root is labeled with $D$ and each node $ q \in \mathcal{T}$ satisfies exactly one of the following possibilities  
\begin{itemize}
\item[(i)] $q \in \mathcal{L}( \mathcal{T} ) $ is a leaf of $\mathcal{T}$ and is labeled with a single-element subset $t = 	\lbrace j \rbrace \subseteq D$.
\item[(ii)] $q \in \I ( \mathcal{T} ) \coloneqq \mathcal{T} \setminus \mathcal{L}( \mathcal{T} )$ is an inner node of $\mathcal{T}$ and has exactly two sons $q_1, q_2 \in \mathcal{T}$, for which the corresponding labels $t,\,t_1,\,t_2 \in \operatorname{Pot}(D)\setminus \{ \emptyset \}$ fulfill $t = t_1  ~\dot{\cup}  ~ t_2$.
\end{itemize}
\end{definition}
We show an example of a dimension tree for $d = 4$ in Figure~\ref{fig:dtee}.\\
	\begin{figure}
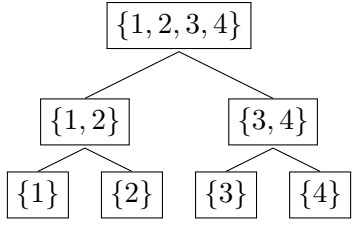

		\Tree [.{$\boxed{ \left\lbrace 1, 2, 3, 4 \right\rbrace }$} [.{$\boxed{ \left\lbrace 1, 2 \right\rbrace }$} {$\boxed{ \left\lbrace 1 \right\rbrace }$} {$\boxed{ \left\lbrace 2 \right\rbrace }$} ] [.{$\boxed{ \left\lbrace  3, 4 \right\rbrace }$} {$\boxed{ \left\lbrace 3 \right\rbrace }$} {$\boxed{ \left\lbrace 4 \right\rbrace }$} ] ] 
		\caption{Dimension tree for $d=4$}\label{fig:dtee}
	\end{figure}
Each node $q\in \mathcal{T}$ represents a non-empty subset $t \subseteq D$ of the modes. 
This leads to the corresponding \emph{matricization} for each node, which we define analogously to~\cite[Definition $3.3$]{Grasedyck2010}:
\begin{definition}[Matricization]\label{def:matricization}
Let $\mathcal{B} \in \mathbb{R}^{\I}$, $t \subseteq D$ with $t \neq \emptyset$, and $s \coloneqq D \setminus t$. 
The \emph{matricization} of $\mathcal{B}$ corresponding to $t$ is defined as $\mathcal{B}^{(t)} \in \mathbb{R}^{\I_t \times \I_s}$, where $\I_t \coloneqq \times_{\mu \in t} \I_\mu $ and $\I_s \coloneqq \bigtimes_{\mu \in s} \I_\mu $, with
\begin{align*}
\mathcal{B}^{(t)} \left[  (i_j)_{j \in t},(i_j)_{j \in s} \right] := \mathcal{B} \left[ i_1,\dots,i_d \right]\qquad \forall ~ i=(i_j)_{j \in D} ~.
\end{align*}
In particular $\mathcal{B}^{(D)} \in \mathbb{R}^{\I}$ holds.
\end{definition}
A matricization corresponds vividly to an unfolding of the tensor as illustrated in Figure~\ref{fig:matricization}.

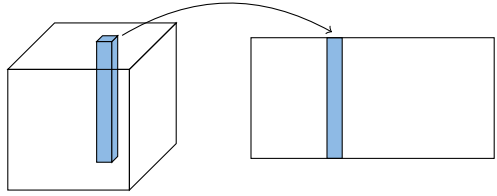
\begin{figure}
	\centering
		\begin{tikzpicture}[scale=4.0]
			\pgfmathsetmacro{\cubex}{0.4}
			\pgfmathsetmacro{\cubey}{0.4}
			\pgfmathsetmacro{\cubez}{0.4}
			\pgfmathsetmacro{\smallcubex}{0.05}
			\pgfmathsetmacro{\smallcubey}{0.05}
			\pgfmathsetmacro{\smallcubez}{0.05}
			\pgfmathsetmacro{\largecubex}{0.8}
			\pgfmathsetmacro{\largecubey}{0.8}
			\pgfmathsetmacro{\largecubez}{0.8}

			\draw[black] (0,0,0) -- ++(0,0,-\cubez) -- ++(0,-\cubey,0) --
		++(0,0,\cubez) -- cycle;
			\draw[black] (0,0,0) -- ++(-\cubex,0,0) -- ++(0,0,-\cubez) --
		++(\cubex,0,0) -- cycle;
		
			\draw[black,fill=rwth-50] (0,0.15,0.15) -- ++(-\smallcubex,0,0) --
		++(0,-\cubey,0) -- ++(\smallcubex,0,0) -- cycle;
			\draw[black,fill=rwth-50] (0,0.15,0.15) -- ++(0,0,-\smallcubez) --
		++(0,-\cubey,0) -- ++(0,0,\smallcubez) -- cycle;
			\draw[black,fill=rwth-50] (0,0.15,0.15) -- ++(-\smallcubex,0,0) --
		++(0,0,-\smallcubez) -- ++(\smallcubex,0,0) -- cycle;
		
			\draw[black] (0,0,0) -- ++(-\cubex,0,0) -- ++(0,-\cubey,0) --
		++(\cubex,0,0) -- cycle;
			\draw[black,fill=rwth-50] (0.7,0.105,0) -- ++(-\smallcubex,0,0) --
		++(0,-\cubey ,0) -- ++(\smallcubex,0,0) -- cycle;
			\draw[black] (1.2,0.105,0) -- ++(-\largecubex,0,0) -- ++(0,-\cubey,0) --
		++(\largecubex,0,0) -- cycle;

			\node(end) at (0.7,0.105){};
			\node(start) at  (0.0, 0.15, 0.15){};
			 \draw[black, ->] (start)  edge [bend left,  above] (end);
		
		\end{tikzpicture}
	\caption{Matricization}
	\label{fig:matricization}
\end{figure}
Based on the concept of matricizations, we define the \emph{hierarchical Tucker rank}, accordingly to~\cite[Definition $3.4$]{Grasedyck2010}:
\begin{definition}[Hierarchical Tucker rank]\label{def:HTRank}
Let $\mathcal{B} \in \mathbb{R}^{\I}$ and $ \mathcal{T} $ be a dimension tree. 
The \emph{hierarchical Tucker rank} of $\mathcal{B}$ is defined as 
		\begin{align*}
			\operatorname{rank}_{ \mathcal{T} }(\mathcal{B}) \coloneqq (r_t)_{t \in \mathcal{T}},
		\end{align*}
where $	r_t \coloneqq \operatorname{rank}( \mathcal{B}^{(t)} )$ denotes the matrix rank of the matricization $\mathcal{B}^{(t)}$ for all $t \in \mathcal{T}$. 
The set of tensors with hierarchical Tucker rank node-wise bounded by $(r_t)_{t \in \mathcal{T}}$ is defined as
\begin{align*}
\HT \left( \mathcal{T} , (r_t)_{t \in \mathcal{T}} \right) \coloneqq \{ ~  \mathcal{C}  \in \mathbb{R}^{\I} ~ \vert ~ \operatorname{rank}( \mathcal{C} ^{(t)}) \leq r_t \quad \forall ~ t \in \mathcal{T} ~ \}~.
\end{align*}
\end{definition}
By construction the so called \emph{nestedness property}
\begin{equation}\label{eq:nessedness}
\begin{aligned} 
&\operatorname{span}\lbrace \mathcal{B}^{(t)}[\cdot, i] \mid 1 \leq i \leq r_t \rbrace \subseteq\\
&\quad \operatorname{span} \lbrace \mathcal{B}^{(t_1)}[\cdot, i_1] \otimes \mathcal{B}^{(t_2)}[\cdot, i_2] \mid 1 \leq i_j \leq r_{t_j}, j=1,2 \rbrace 
\end{aligned}
\end{equation}
holds for all $t \in \mathcal{I}( \mathcal{T} )$ with sons $t_1, t_2$. 
\begin{definition}[(Nested) generator]\label{def:nestedGenerator}
Let $\mathcal{B} \in \mathbb{R}^{\I}$, $\mathcal{T}$ be a dimension tree and $r_t \in \mathbb{N}$ for all $t \in \mathcal{T}$. 
A family of matrices $(U_t)_{t \in \mathcal{T}}$ (also called a \emph{frame tree}) with \emph{frames} $U_t = (U_t[\cdot, 1] \vert \dots \vert U_t[\cdot, r_t] ) \in \mathbb{R}^{ \I_t \times r_t}$ is called a \emph{generator} of $\mathcal{B}$, if the following holds
\begin{align*}
\operatorname{range}(\mathcal{B}^{(t)}) \subseteq \operatorname{range}(U_t) \qquad \forall ~ t \in T~.
\end{align*}
The $(U_t)_{t \in \mathcal{T}}$ are called \emph{nested}, if for all $t \in \mathcal{I}(\mathcal{T})$ with $\operatorname{sons}(t) = \{t_1, t_2\}$
\begin{align*}
U_t\left[\cdot, i\right] \in \operatorname{span} \left\lbrace  U_{t_1}\left[\cdot, i_1\right] \otimes U_{t_2}\left[\cdot, i_2 \right]~ \vert   ~1 \leq i_j\leq r_{t_j}, j=1,2 \right\rbrace 
\end{align*}
holds for all $i \in \{1, \dots, r_t\}$.
\end{definition}
The nestedness property~\eqref{eq:nessedness} allows us to represent the tensor in an efficient way, similarly to~\cite[Definition $3.5$]{Grasedyck2010}, as we only need the range of all matricizations.
\begin{lemma}\label{lemma:nestedness}
Let $\mathcal{B} \in \mathbb{R}^{ \I }$, $\mathcal{T}$ be a dimension tree, $(r_t)_{t \in \mathcal{T}}$ the hierarchical Tucker rank of $\mathcal{B}$ and $t \in \mathcal{I}(\mathcal{T})$ with $\operatorname{sons}(t) = \{t_1, t_2\}$. 
Let further $U_s = (U_s[\cdot, 1] ~ \vert \dots \vert  ~ U_s[\cdot, r_s] ) \in \mathbb{R}^{ \I_s \times r_s}$ be a matrix, which contains column by column a basis of $\operatorname{range}(\mathcal{B}^{(s)})$ for $s \in \{t, t_1, t_2\}$.	
Then there exist coefficients $B_t[i, i_1,i_2] \in \mathbb{R}$, such that
\begin{align*}
U_t[\cdot, i] = \sum\limits_{i_1 = 1}^{r_{t_1}} \sum\limits_{i_2 = 1}^{r_{t_2}} B_t[i,i_1,i_2] ~ \left( U_{t_1}[\cdot, i_1] \otimes U_{t_2}[\cdot, i_2] \right)
\end{align*}
holds. 
The corresponding tensor $B_t \in \mathbb{R}^{r_t \times r_{t_1} \times r_{t_2}}$ is called \emph{transfer tensor}. 
\end{lemma}
Using the nestedness property~\eqref{eq:nessedness} we can represent $\mathcal{B}$ by providing the transfer tensors $B_t$ for all $t\in\mathcal{I}(\mathcal{T})$ and the frames $U_t$ for all $t \in \mathcal{L}(\mathcal{T})$.
The matrices $U_t$ for $t \in \mathcal{T}$ can be calculated, e.g., through the singular value decomposition applied to the corresponding matricizations $\mathcal{B}^{(t)}$.

Moreover, if the frames $U_s$ with $s \in \{t, t_1, t_2\}$ contain orthonormal bases of the range of the corresponding matricizations, the transfer tensor is given by
\begin{align*}
B_t[i, i_1,i_2] = ~ \langle U_t[\cdot, i] ~, U_{t_1}[\cdot, i_1] \otimes U_{t_2}[\cdot, i_2] \rangle ,
\end{align*}
where $\langle \cdot , \cdot \rangle$ denotes the Euclidean scalar product.

We can now define the \emph{hierarchical Tucker format representation} of a tensor similarly to~\cite[Definition $3.6$]{Grasedyck2010}.
\begin{definition}[Hierarchical Tucker format]\label{def:HTFormat}
Let $\mathcal{B} \in \mathbb{R}^{\I}$, $\mathcal{T}$ be a dimension tree, $r_t \in \mathbb{N}$ for all $t \in \mathcal{T}$ with $r_D = 1$, $(U_t)_{t \in \mathcal{L} (\mathcal{T})}$ with $U_t\in \mathbb{R}^{\I_t \times r_t}$ a nested generator of $\mathcal{B}$, and $(B_t)_{t \in \mathcal{I}(\mathcal{T}) }$ the corresponding transfer tensors. 
Then we call
\begin{align*}
\left((U_t)_{t \in \mathcal{L}\left(\mathcal{T}\right) }, ~ (B_t)_{t \in \mathcal{I}(\mathcal{T})} \right)
\end{align*}
a \emph{hierarchical Tucker representation} of $\mathcal{B}$.
The vector $(r_t)_{t \in \mathcal{T}}$ is called \emph{representation rank}.
\end{definition}
The memory required for a hierarchical Tucker representation of a tensor can be specified as follows.
\begin{lemma} \label{lemma:storageCosts}
Let $\mathcal{B} \in \mathbb{R}^{\I}$ with $\I = \bigtimes_{\mu = 1}^d \I_\mu $ and $\vert \I_\mu \vert = n_\mu$ for all $\mu \in D$. 
Let further $\mathcal{T}$ be a dimension tree and $(r_t)_{t \in \mathcal{T}}$ a representation rank of $\mathcal{B}$. 
Then the storage cost for the hierarchical Tucker representation of $\mathcal{B}$ is given by
\begin{align*}
\sum\limits_{\substack{t \in \mathcal{L}\left(\mathcal{T}  \right) \\ t=\{ \mu \}}} n_{\mu} r_{\mu} ~+ \sum\limits_{\substack{t \in \mathcal{I}(\mathcal{T}) \\ \operatorname{sons}(t)=\{t_1,t_2\} }} r_{t} r_{t_1} r_{t_2} ~.
\end{align*}
For $n= \max_{\mu \in D} n_{\mu}$ and $r= \max_{t \in \mathcal{T}} r_t$ the storage cost is in $\mathcal{O} ( rd n+ r^3 d )$.
\end{lemma}
\begin{proof}
See~\cite[Lemma $3.7$]{Grasedyck2010}.
\end{proof}
In~\cite{Grasedyck2010} the existence of a quasi-optimal truncation of a tensor $\mathcal{B} \in \HT (\mathcal{T}, (r_t)_{t \in \mathcal{T}})$ down to lower rank $(\tilde{r}_t)_{t \in \mathcal{T}}$ with an arithmetic cost in $\mathcal{O}( r^2 d n+ r^4 d )$ for $n = \max_{\mu \in D} n_{\mu} $ and $r = \max_{t \in \mathcal{T}} r_t$ was proven. 
For the error of $\tilde{ \mathcal{B} } \in \HT(\mathcal{T}, (\tilde{r}_t)_{t \in \mathcal{T}})$ the quasi-optimal error estimation
\begin{align*}
\Vert \mathcal{B} - \tilde{\mathcal{B}} \Vert \leq \sqrt{2 d - 3} \inf\limits_{ \mathcal{C} \in \HT (\mathcal{T}, (\tilde{r}_t)_{t \in \mathcal{T}})} \Vert  \mathcal{B}  -  \mathcal{C}  \Vert
\end{align*}
holds.

Moreover it is possible to transfer a CP representation with CP rank $r$ of a tensor into a hierarchical Tucker representation with rank node-wise bounded by $r$, cf.~\cite[Theorem $11.17$]{Hackbusch2012}. 
We could therefore also represent the operator and the \righthandside\ in the hierarchical Tucker format.
Inspired by~\cite[Chapter $13$]{Hackbusch2012} we want to summarize some arithmetic operations in the hierarchical Tucker format and their costs in Table~\ref{tab:tensor-arithmetic}.
\begin{table}[!ht]
	\centering
	\caption{ Operations and their costs in the hierarchical Tucker format }\label{tab:tensor-arithmetic}
	\begin{tabular}{lll}
		\toprule
		Operation & Cost & Reference \\ 
		\midrule
		Storage & $\mathcal{O}( dr^3 + d n r )$ & \cite[Lemma $3.7$]{Grasedyck2010}   \\
		Orthonormalization &$\mathcal{O}( 2 d n r^2 + 4 d r^4 )$&\cite[$(13.16b)$]{Hackbusch2012}\\
		Addition & $\mathcal{O}( 8 d n r^2 + 8 d r^4 )$ & \cite[$13.1.4$]{Hackbusch2012} \\
		Evaluation &  $\mathcal{O}( 2dr^3 )$ & \cite[$13.2.3$]{Hackbusch2012} \\ 
		Inner product&  $\mathcal{O}( 2dnr^2 + 6 dr^4 )$ & \cite[Lemma $13.7$]{Hackbusch2012} \\
		Operator application & $\mathcal{O}( 2 d n^2 r )$   & \cite[$13.9.1$]{Hackbusch2012} \\
		Truncation  &  $\mathcal{O}( 2 d r^2 n + 3 d r^4 )$ & \cite[$(11.46c)$]{Hackbusch2012}\\ 
		\bottomrule
	\end{tabular}
\end{table}
	\subsection{Smoother}\label{sec:expSumJac}
In this section we establish a \parameterdependent\ \lowrank\ tensor representation of the damped Jacobi smoother for the multigrid method.

In the case of the Richardson method, the iteration matrix has a CP representation with rank $d + 2$, since we can represent the operator $\mathcal{A}_{\ell}$ and the identity in the CP format with rank $d + 1$ and $1$ as
\begin{equation*}
\mathcal{S}_{\text{Rich},\omega, \ell} =  \bigotimes\limits_{\mu= 0}^d \Id_{n_{\mu}} - \sum\limits_{\nu= 0}^d  \bigotimes\limits_{\mu= 0}^d \omega A_{\ell}^{(\nu)}\left( \mu \right) .
\end{equation*}
In~\cite{Grasedyck2017} we used the damped Richardson method as smoother in a \parameterdependent\ multigrid method using \lowrank\ formats. 
We now want to consider the Jacobi method.

For an efficient Jacobi method we need a \lowrank\ representation of the inverse of the diagonal of $\mathcal{A}_{\ell}$ denoted by $\mathcal{D}_{\ell}^{-1} \coloneqq \diag(\mathcal{A}_{\ell})^{-1}$.
Since we know a CP representation of $\mathcal{A}_{\ell}$ with representation rank $d + 1$, we also know a CP representation of $\mathcal{D}_{\ell}$.
As mentioned in Section~\ref{subsec:CPDecomposition}, the CP format is not closed and, in general, one cannot expect to find an exact CP decomposition of the inverse. 
We thus want to find a sufficiently accurate approximation of the inverse.
Because $\mathcal{D}_{\ell}$ has a CP representation, we want to find an approximation of the inverse, again as a sum of separable elementary tensors and therefore as CP representation.
In the following we want to approximate the inverse of $\mathcal{D}_{\ell}$ with help of exponential sums, using results from~\cite{Hackbusch2019}.

To illustrate the idea, we first take a look at an approximation of $\frac{1}{x + y}$ by exponential sums.
We can approximate the function $x \mapsto \frac{1}{x}$ for $ x \in [1, R]$ by exponential sums through
	\begin{align*}
		\frac{1}{x}	\approx E_{k}(x) \coloneqq  \sum\limits_{m= 1}^k \alpha_{m} \exp(- \beta_{m} x)
	\end{align*}
with weights $ \alpha_{m}, \beta_{m} \in \R_+$.
Hackbusch~\cite{Hackbusch2019} was able to calculate weights corresponding to the interval $[1, R]$ and the number of summands $k \in \N$, such that the approximation $E_k^*$ fulfills an $L^{\infty}$-approximation property with error
	\begin{align*}
		 \epsilon_{\left[1, R\right]}(k) \coloneqq \min\limits_{E_k} \Vert \frac{1}{ \cdot } - E_{k}(\cdot ) \Vert_{\infty, \left[1, R\right]}  = \Vert \frac{1}{ \cdot } - E_{k}^*(\cdot ) \Vert_{\infty, \left[1, R\right]}~ .
	\end{align*}
Using this for the approximation of $\frac{1}{x + y}$ we obtain
	\begin{equation*}
	\begin{aligned}
	\frac{1}{x + y }	 &  \approx \sum\limits_{m= 1}^k \alpha_{m} \exp(- \beta_{m} \left(x + y \right)) 	\\
			&  = \sum\limits_{m= 1}^k \alpha_{m} \exp(- \beta_{m} x) \exp(- \beta_{m} y)
			\end{aligned}
	\end{equation*}
and therewith an approximation of the inverse of a separable sum again as a separable sum.

As a next step, we want to transfer such an exponential sum approximation to the inverse of $\mathcal{D}_{\ell}$.
In this case, the summands are elementary tensors and hence matrices.
For matrices $A, B \in \R^{M \times M}$ the fundamental property of the exponential function $\exp( A + B ) = \exp(A) \exp(B)$ holds, if $A$ and $B$ commute, i.e., $A B = B A$.
Since the diagonal of $\mathcal{A}_{\ell}$ is given through
	\begin{align*} 
		\mathcal{D}_{\ell} = \sum\limits_{\nu = 0}^d \bigotimes\limits_{\mu= 0}^d \diag \left( A_{\ell}^{(\nu)}\left( \mu \right) \right),
	\end{align*}
the single summands commute pairwise, such that the fundamental property of the exponential function holds.

In the following theorem we summarize some conditions needed to approximate the diagonal of a more general CP operator by exponential sums within the CP format.
\begin{theorem}\label{th:expSumApproximation}
		Let $\mathcal{B} = \sum_{\nu = 1}^r \bigotimes_{\mu= 1}^d B^{(\nu)}( \mu ) \in \R^{\I \times \I}$ be a CP operator with spectrum $\sigma ( \diag ( \mathcal{B})) \subseteq [1, R]$ for some $1 < R \leq \infty$.
		Further assume the diagonals of all $B^{\nu}( \mu )$ with $\nu \neq \mu$ to be constant with
			\begin{align*}
				\diag \left( B^{(\nu)}\left( \mu \right)  \right) = b_{\nu, \mu} \cdot \Id \quad \forall ~ \nu \neq \mu ~.
			\end{align*}
		Then for any $k \in \N$ and weights $\alpha_m, \beta_m \in \R_+$ from~\cite{Hackbusch2019} the approximation
			\begin{align*}
			 E_{k}^* \left(\diag \left(\mathcal{B}\right) \right) \coloneqq
				\begin{cases}
				\sum\limits_{m = 1}^k \alpha_{m} \bigotimes\limits_{\nu= 1}^r \exp\left(- \beta_m H^{(\nu)} \right)  \otimes \bigotimes\limits_{\nu = r+1}^d \Id & \text{ if }r < d,\\
				\sum\limits_{m = 1}^k \alpha_{m} \bigotimes\limits_{\nu= 1}^d \exp\left(- \beta_m H^{(\nu)}  \right)  & \text{ otherwise,}
				\end{cases}
		\end{align*}
		with $H^{(\nu)}  \coloneqq ( \prod_{j = 1}^d b_{\nu, j} ) \diag ( B^{(\nu)}( \nu )  ) $, fulfills
			\begin{align*}
				\lVert \left(\diag \left(\mathcal{B}\right) \right)^{-1} -  E_{k}^*\left(\diag \left(\mathcal{B}\right) \right) \rVert_2 \leq \epsilon_{\left[1, R\right]}(k) .
			\end{align*}
\end{theorem}
\begin{proof}
	It holds
		\begin{align*}
			\diag \left(\mathcal{B} \right)  = \sum\limits_{\nu = 1}^r \bigotimes\limits_{\mu= 1}^d  \diag \left( B^{(\nu)}\left( \mu \right)  \right) 
			 = \sum\limits_{\nu = 1}^r \bigotimes\limits_{\mu < \nu} \Id \otimes \left( H^{(\nu)} \right) \otimes \bigotimes\limits_{\mu > \nu} \Id 
		\end{align*}
	and since all summands of the CP representation commute pairwise, the approximation via exponential sums leads to
	\begin{align*}
			E_{k}^*\left(\diag \left(\mathcal{B}\right) \right) &= \sum\limits_{m = 1}^k \alpha_{m} \exp \left( - \beta_m \diag(\mathcal{B}) \right)\\
			&=  \sum\limits_{m = 1}^k \alpha_{m} \prod\limits_{\nu = 1}^r \exp \left( - \beta_m \bigotimes\limits_{\mu < \nu} \Id \otimes \left(H^{(\nu)} \right) \otimes \bigotimes\limits_{\mu > \nu} \Id \right)\\
			&=	
			\begin{cases}
			\sum\limits_{m= 1}^k \alpha_{m} \bigotimes\limits_{\nu= 1}^r \exp\left(- \beta_m H^{(\nu)}  \right)  \otimes \bigotimes\limits_{\nu = r+1}^d \Id  ~ \text{ if }r < d,\\
			\sum\limits_{m = 1}^k \alpha_{m} \bigotimes\limits_{\nu= 1}^d \exp\left(- \beta_m H^{(\nu)} \right) \qquad \qquad \text{ otherwise}.
			\end{cases}
		\end{align*}
\end{proof}
\begin{remark}\label{remark:weights}
Following an idea of~\cite{Hackbusch2019}, we can relax the requirements of Theorem~\ref{th:expSumApproximation} concerning the interval $[1, R]$ demanding the condition $\sigma ( \diag ( \mathcal{B})) \subseteq [a, b]$ for some $0 < a < b \leq \infty$ by scaling the weights like
		\begin{align*}
			\alpha_{m, \left[a,b\right]} = \frac{\alpha_{m, \left[1,R\right]}}{a} \quad \text{and} \quad
			\beta_{m, \left[a,b\right]} = \frac{\beta_{m, \left[1,R\right]}}{a}
		\end{align*}
	for all $m \in \{1, \dots, k\}$, where $R = \frac{b}{a}$ and $\epsilon_{[a, b]}(k) =\frac{ \epsilon_{[1, R]}(k)}{a}$ holds.
\end{remark}
It holds $\mathcal{A}_{\ell} = \sum_{\nu = 0}^d \bigotimes_{\mu= 0}^d A_{\ell}^{(\nu)}( \mu )$ with
	\begin{align*} 
	A_{\ell}^{(\nu)}\left( \mu \right) = 
	\begin{cases} 			
	A_{\ell}^{(\nu)}& \text{if } \mu = d,\\
	\operatorname{diag}\left(p^{(\nu)}\right)& \text{if } \mu+ \nu = d  \text{ and } \nu \neq 0, \\
	\Id_{n_{d-\nu}}& \text{otherwise,}
	\end{cases}
	\end{align*}
and as all $A_{\ell}^{(\nu)}$ are local stiffness matrices, defined, e.g., as in Corollary~\ref{corol:aff1D} or~\ref{corol:aff2D}, their diagonals are either part of stiffness matrices or zero.

Furthermore, the diagonal matrices $\diag(A_{\ell}^{(\nu)})$ for $\nu \in \{1, \dots, d\}$ can be decomposed based on their piecewise different entries $c_{\nu,\gamma} > 0$, such that
\begin{align}\label{eq:decompDiag}
\diag(A_{\ell}^{(\nu)}) = \sum\limits_{\gamma = 1 }^{L_{\nu}} c_{\nu,\gamma}  \widetilde{\Id}_{\nu, \gamma}
\end{align}
holds, where $L_{\nu} \in \N$ is the number of piecewise different entries and diagonal matrices $\widetilde{\Id}_{\nu, \gamma} \in \{0,1\}^{n_{0}  \times n_{0}}$.
With this equation and $L_0 \coloneqq 1$ it holds:
\begin{align}\label{eq:diagSum}
\diag(\mathcal{A}_{\ell}) = \sum\limits_{\nu = 0}^d \sum\limits_{\gamma_{\nu} = 1}^{L_{\nu}} \bigotimes\limits_{\mu= 0}^d \widetilde{A}_{\ell}^{(\nu)}\left( \mu, \gamma_{\nu} \right),
\end{align}
with
\begin{align*} 
\widetilde{A}_{\ell}^{(\nu)}\left( \mu, \gamma_{\nu} \right) = 
\begin{cases} 			
\diag(A_{\ell}^{(0)})& \text{if } \mu = d \text{ and } \nu=0,\\
\widetilde{\Id}_{\nu,\gamma} & \text{if } \mu = d \text{ and } \nu \neq 0,\\
c_{\nu,\gamma} \operatorname{diag}\left(p^{(\nu)}\right)& \text{if } \mu+ \nu = d  \text{ and } \nu \neq 0, \\
\Id_{n_{d-\nu}}& \text{otherwise.}
\end{cases}
\end{align*}
Now we are able to formulate an approximate inverse of our operator via exponential sums.
For the sake of simplicity, we will assume $L_{\nu} = 1$ for all $\nu \in \lbrace 0, \dots, d \rbrace$.
We further assume that the partition $(\Omega_{\nu})_{\nu \in \{1 \dots, d\}}$ of $\Omega$ is so disjoint that
\begin{align}\label{eq:disjointAssumption}
\diag(A^{(\nu_1)}_{\ell}) \cdot \diag(A^{(\nu_2)}_{\ell}) = 0 \text{ for all } \nu_1 \neq \nu_2 \text{ and } \ell >0
\end{align}
holds true.
This assumption is valid, e.g., for our model problem~\eqref{eq:cookie} if it is discretized as in Theorems~\ref{th:cookie1DStencil} or~\ref{th:cookie2DStencil} and the edges of the grid cells of the piecewise different parameters on the coarsest grid have positive distance.
Such an assumption is invalid, e.g., in case of intersecting parameters, which occur through the Karhunen-Lo{\`e}ve expansion.
In future work, we want to generalize our results for such problems. 
\begin{theorem}\label{th:approximativeExpSumInverse}
Let $\sigma ( \mathcal{D}_{\ell} ) \subseteq [a, b]$ for some $0 < a < b \leq \infty$, $L_{\nu} = 1$ for all $\nu \in \lbrace 0, \dots, d \rbrace$, assumption~\eqref{eq:disjointAssumption} holds true, $\ell, k \in \N$ and $\alpha_m, \beta_m \in \R_+$ be weights for the exponential sum approximation from Remark~\ref{remark:weights}.
Then $E_k^*( \mathcal{D}_{\ell} )$ has a CP representation given by
\begin{align*}
E_k^*\left( \mathcal{D}_{\ell} \right) = \sum\limits_{m = 1}^k \sum\limits_{\nu=1 }^d \alpha_{m} \bigotimes\limits_{\mu = 0}^{d} \widehat{D}_{\ell}^{(m, \mu, \nu)} ~,
\end{align*}
where 
\begin{align*}
\widehat{D}_{\ell}^{(m, \mu, \nu)} = 
\begin{cases}
	\exp \left(- \beta_m \diag\left(A_{\ell}^{(0)}\right)\right) \widetilde{\Id}_{\nu} & \text{if } \mu = d,\\
	\exp \left( -\beta_{m} c_{d -\mu} D_{\ell}^{(d - \mu)} \right) & \text{if } \mu = \nu,\\
	\Id_{n_{d-\mu}} & \text{otherwise}.
\end{cases}
\end{align*}
\end{theorem}
\begin{proof}
Since all summands in the CP representation of $\mathcal{D}_{\ell}$ are diagonal matrices they commute pairwise, which allows the factorization of the exponential function as
\begin{align*}
E_k^*\left( \mathcal{D}_{\ell} \right) = \sum\limits_{m = 1}^k \alpha_m \prod\limits_{\nu=0}^d \prod\limits_{\gamma_{\nu}=1}^{L_{\nu}} \exp \left( \bigotimes\limits_{\mu=0}^d \widetilde{A}_{\ell}^{(\nu)}\left( \mu, \gamma_{\nu} \right) \right).
\end{align*}
Moreover for any quadratic matrix $M$ and $\widetilde{\Id}$ as in equation~\eqref{eq:decompDiag} it holds
\begin{align*}
\exp\left( \Id \otimes M \otimes \widetilde{\Id} \right) = \Id \otimes \exp \left( M \right) \otimes \widetilde{\Id} + \Id \otimes \Id \otimes \left( \Id - \widetilde{\Id} \right) ,
\end{align*}
since $( \widetilde{\Id} )^j = \widetilde{\Id}$ for all $j \in \N$.
With $L_{\nu}=1$ we derive
\begin{align*}
E_k^*\left( \mathcal{D}_{\ell} \right) = \sum\limits_{m = 1}^k \sum\limits_{t\subseteq \left\lbrace 1,\dots,d \right\rbrace } \alpha_{m} \bigotimes\limits_{\mu = 0}^{d} \widehat{D}_{\ell}^{(m, \mu, t)} ~,
\end{align*}
where 
\begin{align*}
\widehat{D}_{\ell}^{(m, \mu, t)} = 
\begin{cases}
	\exp \left(- \beta_m \diag\left(A_{\ell}^{(0)}\right)\right) \widetilde{\Id}_t & \text{if } \mu = d,\\
	\exp \left( -\beta_{m} c_{d -\mu} D_{\ell}^{(d - \mu)} \right) & \text{if } \mu \in t,\\
	\Id_{n_{d-\mu}} & \text{otherwise},
\end{cases}
\end{align*}
with $\widetilde{\Id}_t \coloneqq \prod_{\nu \in t} \widetilde{\Id}_{\nu} \prod_{\eta \in t^c} ( \Id_{n_{0}} - \widetilde{\Id}_{\eta} )$.
With assumption~\eqref{eq:disjointAssumption} we obtain $\widetilde{\Id}_t =0$ for $\vert t \vert \geq 2$ and therefore
\begin{align*}
\widetilde{\Id}_{ \left\lbrace \nu \right\rbrace }  = \widetilde{\Id}_{\nu} \prod\limits_{\substack{\eta = 1\\ \eta \neq \nu}}^d \left( \Id_{n_0} - \widetilde{\Id}_{\eta} \right) = \prod\limits_{\substack{ \eta = 1\\ \eta \neq \nu}}^d \left(  \widetilde{\Id}_{\nu} -  \widetilde{\Id}_{\nu} \widetilde{\Id}_{\eta} \right) =  \widetilde{\Id}_{\nu}.
\end{align*}
Thus, the theorem is true.
\end{proof}
\begin{remark}
The representation rank of the inverse diagonal from Theorem~\ref{th:approximativeExpSumInverse} is bounded by $k\cdot d$.
Since $k$ derives from the approximation by exponential sums and can be uniformly bounded for all parameter values, the rank of the inverse of the diagonal of the operator grows linearly in the number of parameters.
\end{remark}
With the result of Theorem~\ref{th:approximativeExpSumInverse} we define the iteration matrix and prove the smoothing property for the approximate damped Jacobi method.
\begin{theorem}\label{th:smoothingPropertyApproximativeJacobi}
Let $\sigma ( \mathcal{D}_{\ell} ) \subseteq [a, b]$ for some $0 < a < b \leq \infty$, $L_{\nu} = 1$ for all $\nu \in \lbrace 0, \dots, d \rbrace$, $k \in \N$ and $\alpha_m, \beta_m \in \R_+$ be weights for the exponential sum approximation from Remark~\ref{remark:weights}.
The iteration matrix of the \emph{approximate damped Jacobi method} is given by
\begin{align*}
\mathcal{S}_{\text{approxJac}, k, \omega, \ell} \coloneqq \Id - \omega E_k^* \left(\mathcal{D}_{\ell}\right) \mathcal{A}_{\ell},
\end{align*}
and fulfills the smoothing property for any damping parameter $\omega \in ( 0, \omega_0  )$, with 
\begin{align*}
\omega_0 =  \frac{1}{\rho \left(E_k^* \left(\mathcal{D}_{\ell}\right) \mathcal{A}_{\ell} \right)} .
\end{align*}
\end{theorem}
\begin{proof}
Since all $\widehat{D}_{\ell}^{(m, \mu, \nu)}$ have positive diagonal entries and all weights $\alpha_m$ are positive, $E_k^* (\mathcal{D}_{\ell})$ is symmetric positive definite. 
Therefore the theorem follows with Lemma~\ref{lemma:smoothingPropertyParameter} similarly to Theorem~\ref{th:smoothingJacobiParameter}.
\end{proof}
Since a high number of different diagonal values in equation~\eqref{eq:diagSum} can lead to an increased representation rank of the approximative inverse, we want to find a relaxed approximation of the inverse independent of $L_\nu$.
Therefore we approximate the diagonal $\mathcal{D}_{\ell}$ of our operator $\mathcal{A}_{\ell}$ by
\begin{align}\label{eq:tildeD}
\widetilde{\mathcal{D}_{\ell}} = \sum\limits_{\nu = 0}^d \bigotimes\limits_{\mu < \nu} \Id_{n_{d - \mu}} \otimes ~  c_{d - \nu} D_{\ell}^{(d - \nu)} \otimes \bigotimes\limits_{\mu > \nu} \Id_{n_{d - \mu}} ,
\end{align}
where $c_{\mu} \coloneqq \max_{i \in \I_{\mu}} A_{\ell}^{(\mu)}[i,i] \geq 0$ for all $\mu \in \{1, \dots, d\}$, $c_0 \coloneqq 1$ and $D^{(0)}_{\ell} \coloneqq \diag (A^{(0)})$, to apply again exponential sums.
For $k \in \N$ and corresponding weights we obtain 
\begin{align*}
E_k^*(\widetilde{\mathcal{D}_{\ell}}) = \sum\limits_{m = 1}^k \alpha_m \bigotimes\limits_{\mu = 0}^{d} \exp \left(- \beta_m c_{d - \mu} D_{\ell}^{(d - \mu)} \right),
\end{align*}
if the spectrum of $ \widetilde{\mathcal{D}_{\ell}}$ is sufficiently bounded.
We prove the smoothing property for this approximation in the following theorem.
\begin{theorem}\label{th:smoothingPropertyModifiedJacobi}
Let $\diag (A_{\ell}^{(\mu)})$ and $D_{\ell}^{(\mu)}$ have only nonnegative entries for all $\mu \in \{0, \dots , d\}$.
Let further $\widetilde{\mathcal{D}_{\ell}}$ be defined as in equation~\eqref{eq:tildeD} with spectrum $\sigma( \widetilde{ \mathcal{D}_{\ell} }) \subseteq [a, b]$ for some $0 < a < b \leq \infty$, $k  \in \N$ and let $\alpha_m, \beta_m \in \R_+$ be weights for the exponential sum approximation from Remark~\ref{remark:weights}.
The iteration matrix of the \emph{modified approximate damped Jacobi method} is given by
\begin{align*}
\mathcal{S}_{\text{modJac}, k, \omega} \coloneqq \Id - \omega E_k^* \left(\widetilde{\mathcal{D}_{\ell}}\right) \mathcal{A}_{\ell} ,
\end{align*}
and fulfills the smoothing property for any damping parameter $\omega \in ( 0, \omega_0  )$, where
\begin{align*}
\omega_0 =  \frac{1}{\rho \left(  E_k^* \left(\widetilde{\mathcal{D}_{\ell}}\right) \mathcal{A}_{\ell}\right)}  .
\end{align*}  
\end{theorem}
\begin{proof}
Analogous to Theorem~\ref{th:smoothingPropertyApproximativeJacobi}.
\end{proof}
	\subsection{Prolongation and restriction}
We want to find a \parameterdependent\ representation of the prolongation and the restriction for our affine operator.
As we want to define the coarser grid using the \Galerkinansatz, we choose the canonical prolongation and restriction as in~\cite{Ballani2013}.
\begin{corollary}\label{cor:prolongationRestriction}
Assume that with
	\begin{align*}
		R_{\ell} = \frac{1}{4}
		\begin{pmatrix}
		1 & 2 & 1 &  &  & & & &\\
		&  & 1 & 2 & 1 & & & &\\
		&  &  &  & \ddots &  &  & & \\
		&  &    &   & & & 1 & 2 & 1
		\end{pmatrix}
		\text{ and } P_{\ell} = 2 R^T_{\ell}~,
	\end{align*}	
	the \Galerkinansatz\ $A_{\ell-1}(p) = R_{\ell} A_{\ell}(p) P_{\ell}$ holds for all $p \in \I$, then $R_{\ell}(p)$ and $P_{\ell}(p)$ have a CP representation of rank~$1$.
\end{corollary}
\begin{proof}
Using the \Galerkinansatz\ we get
		\begin{align*}
			A_{\ell-1}(p) = R_{\ell} A_{\ell}(p) P_{\ell} =   R_{\ell} A_{\ell}^{(0)} P_{\ell} + \sum\limits_{\nu=1}^d p^{(\nu)} R_{\ell} A_{\ell}^{(\nu)} P_{\ell}.
		\end{align*}
The same calculation as for the operator then yields
		\begin{align*}
			R_{\ell}(p) = & \Id _{n_d}\otimes \Id_{n_{d-1}} \otimes \dots \otimes \Id_{n_1} \otimes R_l, \\
			P_{\ell}(p) = & \Id _{n_d}\otimes \Id_{n_{d-1}} \otimes \dots \otimes \Id_{n_1}  \otimes P_l
		\end{align*}
and thus a CP representation of rank~$1$.
\end{proof}
Concluding we have developed all components needed for a \parameterdependent\ multigrid method.
\section{Numerical experiments}\label{sec:NumExp}
We derived \parameterdependent\ representations of the operator, the \righthandside, the prolongation, the restriction and an approximation of the smoother.
Now, we present numerical experiment of the corresponding multigrid method for \parameterdependent\ problems. 
We display the geometry used in our numerical experiments in Figure~\ref{fig:TwoCookieGeometry} and discretize equation~\eqref{eq:cookie} by the finite-difference method.

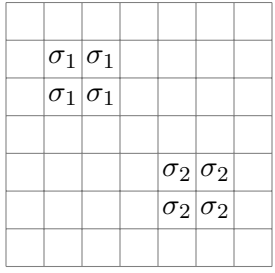
\begin{figure}
\centering   
\begin{tikzpicture}
\draw[step=0.5cm,color=gray] (0,0) grid (3.5,3.5);
\node at (2.25,1.25) {$\sigma_2$};
\node at (2.75,1.25) {$\sigma_2$};
\node at (2.25,0.75) {$\sigma_2$};
\node at (2.75,0.75) {$\sigma_2$};
\node at (0.75,2.25) {$\sigma_1$};
\node at (1.25,2.25) {$\sigma_1$};
\node at (0.75,2.75) {$\sigma_1$};
\node at (1.25,2.75) {$\sigma_1$};
\end{tikzpicture} 
\caption{Geometry of $[\num{0},\num{7}] \times [\num{0},\num{7}]$ used for the numerical experiments}\label{fig:TwoCookieGeometry}
\end{figure}

As mentioned in Corollary~\ref{corol:aff2D}, the operator has an affine structure and therefore we get the following \parameterdependent\ linear system:
\begin{equation}\label{eq:TwoCookieProblem}
\left(A^{(0)}_{\ell} + p^{(1)} A_{\ell}^{(1)} + p^{(2)} A_{\ell}^{(2)} \right) ~ u_{\ell}(p) = f_{\ell} ,
\end{equation}
with $p^{(1)},\,p^{(2)} \in \lbrace \num{0},\num[quotient-mode = fraction]{1/100},\num[quotient-mode = fraction]{2/100}, \dots, \num{1} \rbrace$, \righthandside\ $f_{\ell} \equiv 1$ and $A_{\ell}^{(d)}$ as in Theorem~\ref{th:cookie2DStencil}.
We choose the grid such that the coarsest grid $\ell = 0$ has $\num{7} \times \num{7}$ points and refine this grid to $\num{15} \times \num{15}$, then to $\num{31} \times \num{31}$ and for $\ell = 3$ to $\num{63} \times \num{63}$ points.

We now want to solve equation~\eqref{eq:TwoCookieProblem} with the V-cycle multigrid method using the \parameterdependent\ damped Jacobi method by means of exponential sums from Section~\ref{sec:expSumJac} as smoother.
In our first numerical experiment we want to compare it with the V-cycle multigrid method using the damped Richardson method as smoother and with the \parameterdependent\ damped Jacobi method by means of exponential sums as an iterative solver. 
In the log-log plot, shown in Figure~\ref{fig:ExperimentComparison}, we plot the relative residual of the finest grid error, i.e., $\ell = 3$, against the number of iteration.
In our numerical experiments we use $5$ presmoothing and $5$ postsmoothing steps and we choose the damping factor of $\omega = \num{1e-5}$ for the Richardson method and of $\omega = \num[quotient-mode = fraction]{1 /2}$ for the Jacobi method, since smaller factors seemed to slow down convergence, while the methods with larger damping factors sometimes diverged.
We truncate the representation of the solution after each rank increasing operation using the method described in~\cite{Grasedyck2010} with a tolerance value of $\num{1e-7}$.
\begin{figure}
\centering
\includegraphics[width=0.8\textwidth]{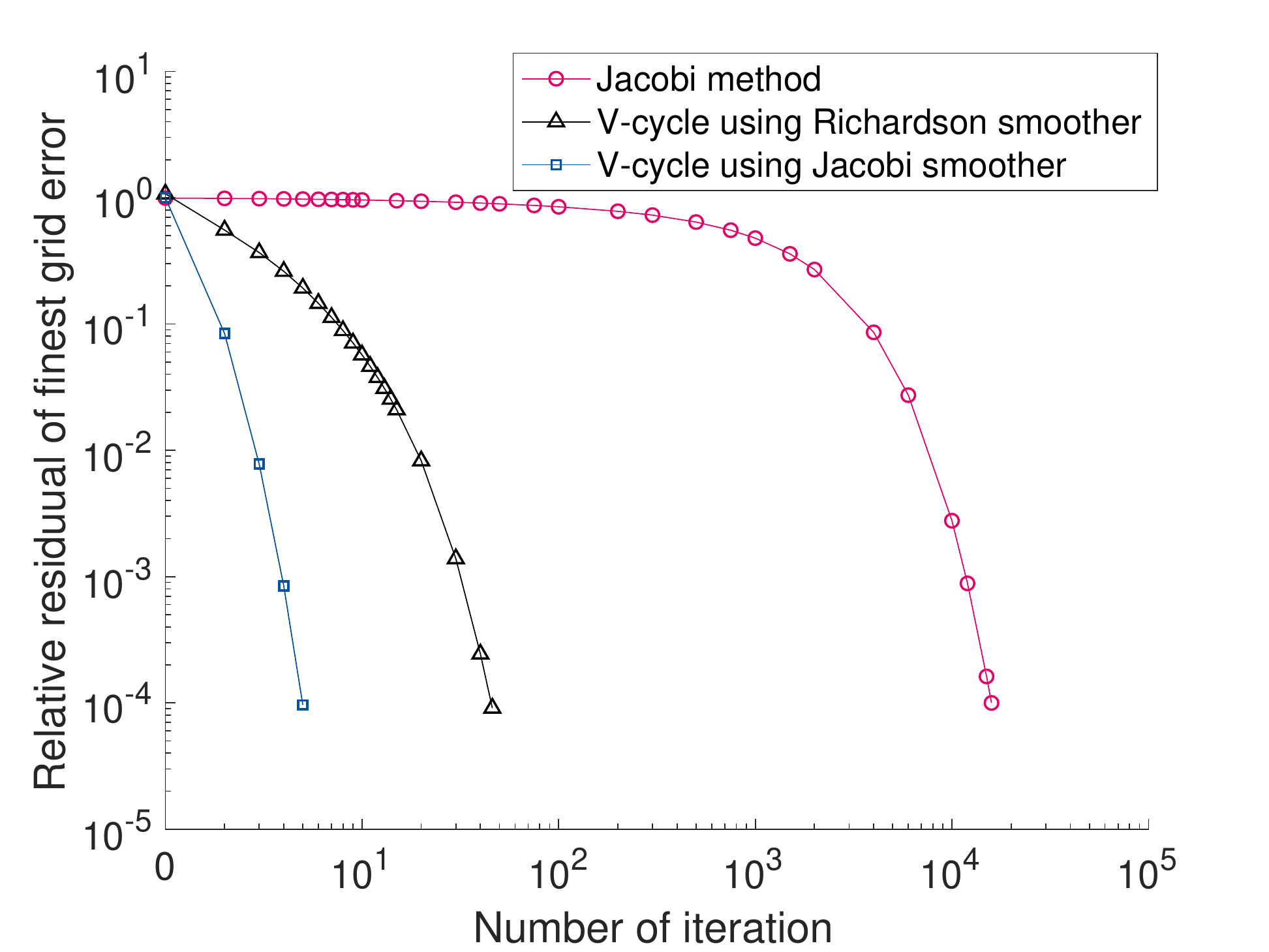}
\caption{Comparison of different solvers of the \parameterdependent\ 
linear system from equation~\eqref{eq:TwoCookieProblem} up to a relative residual of $\epsilon = \num{1e-4}$}\label{fig:ExperimentComparison}
\end{figure}

In Figure~\ref{fig:ExperimentComparison} we observe that the damped Jacobi method as a solver needs $\num{\sim 16000}$ iterations in order to reduce the relative residual to less then the prescribed tolerance of $\epsilon=\num{1e-4}$.
Therefore the damped Jacobi method seems to have a slow convergence behavior.
We also observe that using the damped Richardson method as smoother in a V-cycle multigrid method, we need $\num{\sim 45}$ iterations, and that instead using the damped Jacobi method as smoother in a V-cycle multigrid method, we need $\num{\sim 5}$ iterations to converge to the prescribed tolerance.
Due to this, we observe for the V-cycle multigrid method using the Jacobi method as smoother a faster convergence behavior then using the Richardson method as smoother.

In our next numerical experiment we compare the V-cycle multigrid method using the \parameterdependent\ damped Jacobi method by means of exponential sums from Section~\ref{sec:expSumJac} as smoother for different grid sizes.
In the log-lin plot of Figure~\ref{fig:ExperimentTwoGrid} we plot the relative residual of the finest grid solution against the number of iterations for some grid sizes.
We used the grid of level $\ell = 0$ with $\num{7} \times \num{7}$ points from above as coarsest grid in all $3$ cases.
\begin{figure}
    \centering
     \includegraphics[width=0.8\textwidth]{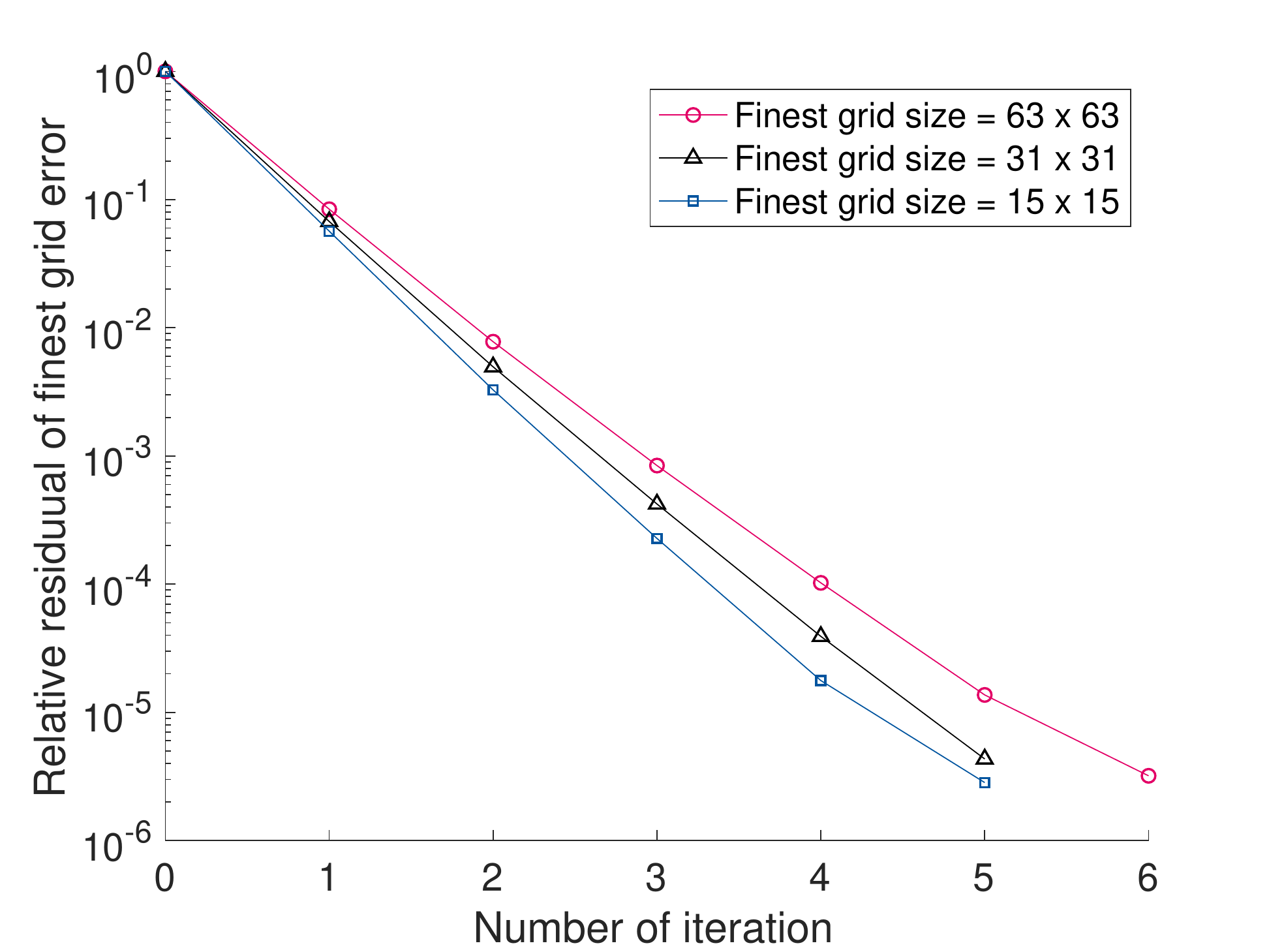}
        \caption{Comparison of the V-cycle multigrid method for different finest grid sizes as solver for the \parameterdependent\ linear system from equation~\eqref{eq:TwoCookieProblem} up to a relative residual of $\epsilon = \num{1e-5}$}\label{fig:ExperimentTwoGrid}
\end{figure}

In Figure~\ref{fig:ExperimentTwoGrid} we observe that the convergence rate of the multigrid method using our damped Jacobi smoother seems to be grid size independent.

In future work, one could use the level-wise parallelism of the hierarchical Tucker format to accelerate the arithmetic operations within the format.
Using a balanced tree allows the reduction of the cost dependency for most operations in Table~\ref{tab:tensor-arithmetic} from $d$ to $\log(d)$, cf.~\cite{Etter2016,Grasedyck2017,Grasedyck2019,Grasedyck2016}.

In summary, given a \parameterdependent\ representation of a linear system, such that the smoothing and approximation properties hold, we can guarantee the multigrid convergence. 
For a model problem, using \lowrank\ tensor formats, we derived such a \parameterdependent\ representation and an approximation of the damped \parameterdependent\ Jacobi method by means of exponential sums.
In numerical experiments we observed a grid size independent convergence rate using the multigrid method with our novel approximation of the damped Jacobi smoother.
\section*{Acknowledgments}%
L.\@ Grasedyck, C.\@ L\"obbert and T.\@ A.\@ Werthmann have been supported by the German Research Foundation (DFG) within the DFG priority programme 1648 (SPPEXA) under Grant No.\@ GR-3179/4-2 and 1886 (SPPPoly) under Grant No.\@ GR-3179/5-1.

M.\@ Klever has been supported by the DFG through the grant SFB/TRR-55.
\printbibliography{}
\end{document}